\documentclass[11pt]{article}
\usepackage{geometry}                
\usepackage{graphicx}
\usepackage{amssymb}
\usepackage{epstopdf}
\usepackage{amsmath,amsthm,amscd,amssymb}
\usepackage{latexsym}
\numberwithin{equation}{section}

\theoremstyle{plain}
\newtheorem{theorem}{Theorem}[section]
\newtheorem{lemma}[theorem]{Lemma}
\newtheorem{corollary}[theorem]{Corollary}
\newtheorem{proposition}[theorem]{Proposition}

\theoremstyle{definition}
\newtheorem{definition}[theorem]{Definition}

\newtheorem{example}[theorem]{Example}

\theoremstyle{remark}

\newtheorem{case[theorem]}{Case}

\def\bb #1{ {\mathbb #1} }

\title{Tiling sets and spectral sets over finite fields }
\author{C. Aten, B. Ayachi, E. Bau, D. FitzPatrick, A. Iosevich, H. Liu, A. Lott, \\ I. MacKinnon, S. Maimon, S. Nan, J. Pakianathan, G. Petridis, \\ C. Rojas Mena, A. Sheikh, T. Tribone, J. Weill, C. Yu}

\begin{document}
\maketitle

\begin{abstract} We study tiling and spectral sets in vector spaces over prime fields. The classical Fuglede conjecture in locally compact abelian groups says that a set is spectral if and only if it tiles by translation. This conjecture was disproved by T. Tao in Euclidean spaces of dimensions 5 and higher, using constructions over prime fields (in vector spaces over finite fields of prime order) and lifting them to the Euclidean setting. Over prime fields, when the dimension of the vector space is less than or equal to $2$ it has recently been proven that the Fuglede conjecture holds (see \cite{IMP15}). In this paper we study this question in higher dimensions over prime fields and provide some results and counterexamples. In particular we prove the existence of spectral sets which do not tile in $\mathbb{Z}_p^5$ for all odd primes $p$ and $\mathbb{Z}_p^4$ for all odd primes $p$ such that $p \equiv 3 \text{ mod } 4$. Although counterexamples in low dimensional groups over cyclic rings $\mathbb{Z}_n$ were previously known they were usually for non prime $n$ or a small, sporadic set of primes $p$ rather than general constructions. This paper is a result of a Research Experience for Undergraduates program ran at the University of Rochester during the summer of 2015 by A. Iosevich, J. Pakianathan and G. Petridis.  \\

2010 {\it Mathematics Subject Classification.}
Primary: ;
Secondary: .
\end{abstract} 

\tableofcontents

\section{Introduction}

\vskip.125in 

The purpose of this paper is to study the relationships between tiling properties of sets and the existence of orthogonal exponential bases for functions on these sets in the context of vector spaces over finite fields. As both tiling properties and spectral properties of sets depend only on the underlying abelian group of these vector spaces, 
it is enough to understand these relationships over prime fields $\mathbb{Z}_p = \mathbb{F}_p$. This is because for any prime $p$, the finite field 
$\mathbb{F}_{p^s}$ is additively isomorphic to $\mathbb{Z}_p^s$ so that $\mathbb{F}_{p^s}^d \cong \mathbb{Z}_p^{ds}$ as abelian groups.
Due to this, in the remainder of this paper, we exclusively consider these questions in vector spaces over prime fields i.e., in $\mathbb{Z}_p^d$, where $p$ is a prime.

The study of the relationship between exponential bases and tiling has its roots in the celebrated Fuglede Conjecture in ${\Bbb R}^d$, which says that if $E \subset {\Bbb R}^d$ of positive Lebesgue measure, then $L^2(E)$ possesses an orthogonal basis of exponentials if and only if $E$ tiles ${\Bbb R}^d$ by translation. Fuglede proved this conjecture in the celebrated 1974 paper (\cite{Fu74}) in the case when either the tiling set or the spectrum is a lattice. A variety of results were proved establishing connections between tiling and orthogonal exponential bases. See, for example, \cite{LRW00}, \cite{IP98}, \cite{L02}, \cite{KL03} and \cite{KL04}. In 2001, Izabella Laba proved the Fuglede conjecture for unions of two intervals in the plane (\cite{L01}). In 2003, Iosevich, Katz and Tao  (\cite{IKT03}) proved that the Fuglede conjecture holds for convex planar domains. 

A cataclysmic event in the history of this problem took place in 2004 when Terry Tao (\cite{T03}) disproved the Fuglede Conjecture by exhibiting a spectral set in ${\Bbb R}^{5}$ which does not tile. The first step in his argument is the construction of a spectral subset of ${\Bbb Z}_3^5$ of size $6$. It is easy to see that this set does not tile $\mathbb{Z}_3^5$ because $6$ does not divide $3^5$. As a by-product, this shows that spectral sets in ${\Bbb Z}_p^d$ do not necessarily tile at least in the cases $p=3, d \geq 5$ and $p=2, d \geq 11$ considered by Tao. See \cite{KM06}, where Kolountzakis and Matolcsi also disprove the reverse implication of the Fuglede Conjecture. In \cite{FR06} and \cite{FMM06}, the dimension of counter-examples was further reduced. In fact, Farkas, Matolcsi and Mora show in \cite{FMM06} that the Fuglede conjecture fails in $3$ dimensions by proving the existence of a tiling set which is not spectral in ${\Bbb Z}_{n}^3$ (where $n$ is a suitably large multiple of $24$) by first constructing a tiling set without a universal spectrum in ${\Bbb Z}_{24}^3$. The general feeling in the field was that sooner or later the counter-examples of both implication will cover all dimensions. However, Iosevich, Mayeli and Pakianathan proved in \cite{IMP15} that the Fuglede Conjecture holds in two-dimensional vector spaces over prime fields. This result is reproved in this paper in the course of surveying results as well as the implication that over prime fields, in dimension 3, tiling sets are always spectral. This is in contrast to the aforementioned examples which show that the implication tiling $\to$ spectral fails in $\mathbb{Z}_{n}^3$ (where $n$ is some multiple of $24$) and hence in $\mathbb{R}^3$. 

In this paper we study this question in higher dimensions and provide some results and counterexamples. In particular we prove the existence of spectral sets which do not tile in $\mathbb{Z}_p^5$ for all odd primes $p$ and $\mathbb{Z}_p^4$ for all odd primes $p$ such that $p \equiv 3 \text{ mod } 4$. Although counterexamples in low dimensional groups over cyclic rings $\mathbb{Z}_n$ were previously known they were usually for non prime $n$ or a small, sporadic set of primes $p$ rather than general constructions. 

This paper is structured as follows. After summarizing some basic facts about tiling sets and spectral sets in the first two sections, we prove various results regarding them and construct some interesting counterexamples in prime fields. We develop further machinery such as that of Davey matrices which allow human-readable 
verifications of the Fuglede conjecture in $\mathbb{Z}_2^3$ and $\mathbb{Z}_3^3$ which is done in the last section.

We also verify certain conditions where the Fuglede conjecture holds in general:

\begin{theorem}
\label{theorem:main1}
Let $E \subseteq \mathbb{Z}_p^d$, $p$ a prime, and let $|E|$ denote the number of elements of $E$. 
\begin{itemize}
\item[(a)] If $E$ is a tiling set then $|E|=p^r$, $0 \leq r \leq d$. 
\item[(b)] If $E$ is a spectral set then $|E|=1, p^d$ or $|E|=kp$ for some $1 \leq k \leq p^{d-2}$. 
\item[(c)] If $E$ is a spectral set in $\mathbb{Z}_2^d$ then $|E|=1,2$ or a multiple of $4$. 
\item[(d)] A set $E$ tiles with a subspace tiling partner $V$ if and only if $E$ is spectral with spectrum $V^{\perp}$. 
This happens if and only if $E$ is a full graph set. 
\item[(e)] If $|E|=p, p^{d-1}$ then $E$ tiles if and only if $E$ is spectral. Furthermore $E$ has a subspace tiling partner and hence is a graph set when this occurs. 
\item[(f)] In dimensions $d \leq 2$, $E$ tiles if and only if $E$ is spectral. Furthermore $E$ is a graph set when this occurs. (First obtained in \cite{IMP15}). 
\item[(g)] In dimension $d=3$, $E$ tiles $\rightarrow$ $E$ spectral. Furthermore $E$ is a graph set when this occurs. 
\item[(h)] If $E$ is any spectral set then $E$ either tiles or $k$-tiles where $k=\frac{|E|}{p}$ with a hyperplane partner. 
\item[(i)] If $E$ is a spectral set of size $mp$ in $\mathbb{Z}_p^d$ then there exists a function 
$f: E \to \mathbb{Z}_m$ such that $\text{Graph}(f) = \{ (e,f(e)) | e \in E \} \subseteq \mathbb{Z}_p^d \times \mathbb{Z}_m$ tiles 
$\mathbb{Z}_p^d \times \mathbb{Z}_m$. Furthermore the projection (forgetting the last coordinate) from 
$\mathbb{Z}_p^d \times \mathbb{Z}_m \to \mathbb{Z}_p^d$ takes $\text{Graph}(f)$ bijectively to $E$. 
Thus every spectral set is the bijective image of a tiling set under a projection. 
\end{itemize}
\end{theorem}

We next construct low-dimensional examples of spectral sets which do not tile for all odd primes. 
In general constructing counter-examples over prime fields is more constrained than over cyclic groups of non-prime order. For example 
though the implication tiling $\to$ spectral is true in $\mathbb{Z}_p^3$ for any prime $p$, it is not true in some $\mathbb{Z}_{n}^3$ when $n$ is not prime.  See \cite{FMM06} for a construction of this type. 

\begin{theorem}
\label{theorem:main2}
Let $p$ be an odd prime then:
\begin{itemize}
\item[(a)] There are examples of spectral sets of size $2p$ in $\mathbb{Z}_p^5$ which do not tile for every odd prime $p$. 
\item[(b)] When $p \equiv 3 \text{ mod } 4$ there are examples of sets of size $2p$ in $\mathbb{Z}_p^4$ which do not tile. 
\end{itemize}
Thus the implication $E$ spectral $\to$ $E$ tiles is always false in $5$ or more dimensions over $\mathbb{Z}_p$, $p$ any odd prime and 
is always true in $2$ or less dimensions for all primes. For $p \equiv 3 \text{ mod } 4$ it is false in dimension $4$ also. 
\end{theorem}

Theorem~\ref{theorem:main1} and \ref{theorem:main2} essentially settle the status of Fuglede's conjecture over prime cyclic rings 
in all dimensions except three. In 3 dimensions, we provide a readable proof that Fuglede's conjecture holds over 
$\mathbb{Z}_2^3$ and $\mathbb{Z}_3^3$ in Theorem~\ref{theorem:3Dverifications} but it remains open for higher primes.

The existence of nonprime counterexamples is not indicative in this regard as the implication tiling $\to$ spectral holds in $\mathbb{Z}_p^3$ 
when $p$ is prime but not in general when $\mathbb{Z}_{n}^3$ is not prime and so the restriction to prime cyclic rings is important and known to make a difference.

\section{Fourier transform and cones}

Let $f: \mathbb{Z}_n^d \to \mathbb{C}$ be a complex-valued function, the Fourier transform $\widehat{f}$ of $f$ is defined via
$$
\widehat{f}(m) = \frac{1}{n^d} \sum_{x \in \mathbb{Z}_n^d} f(x) \chi(-x \cdot m)
$$
where $x \cdot m = x_1m_1 + \dots + x_d m_d$ is the ``dot product'' and $\chi(u) = e^{\frac{2 \pi i u}{n}}$ is the canonical additive character of 
$\mathbb{Z}_n$.

If $\mathfrak{1}$ denotes the constant function with value $1$ then one easily computes that $\widehat{\mathfrak{1}}(m)= \delta(m)$ where $\delta$ is the Kronecker delta function.

A crucial fact about the Fourier transform over prime fields $\mathbb{Z}_p$ when $p$ a prime is the following equidistribution result which can be found in 
\cite{IMP15} for example:

\begin{lemma} Let $p$ be a prime and $E \subseteq \mathbb{Z}_p^d$ and $m \in \mathbb{Z}_p \backslash \{0 \}$. The following are equivalent:  \\
\label{lemma:main} 
\begin{itemize}
\item[(1)] $\widehat{E}(m) = 0$.
\item[(2)] $E$ equidistributes on the $p$ parallel hyperplanes $H_{m,t} = \{ x | x \cdot m = t \}$, $t=0,1,\dots,p-1$. 
\item[(3)] $\widehat{E}(rm)=0$ for all $r \in \mathbb{Z}_p - \{ 0 \}$.
\end{itemize}
\end{lemma}

\vskip.125in 

\begin{definition}
A punctured line in $\mathbb{Z}_p^d$ is a line through the origin with the origin removed.  
A {\bf cone} in $\mathbb{Z}_p^d$ is a set which is a union of punctured lines. Equivalently $C$ is a cone if and only if for every $c \in C$ and $r \in \mathbb{Z}_p-\{ 0 \}$ we have $rc \in C$. 

\vskip.125in 

Some important examples of cones that we will use are the following: \\
For any $E \subseteq \mathbb{Z}_p^d$ we define 
$Z(\widehat{E}) = \{ m | \widehat{E}(m) = 0 \}$ the {\bf zero cone} of $\widehat{E}$ and
$\text{Spt}(\widehat{E}) = \mathbb{Z}_p^d - Z(\widehat{E}) - \{ 0 \}$ the { \bf support cone } of $\widehat{E}$. 
These are cones because of Lemma~\ref{lemma:main}.

We also define the {\bf direction set} of $E$, $\text{Dir}(E)=\{ e_1 - e_2 | e_1 \neq e_2 \in E \}$ and {\bf direction cone} of $E$, 
$\text{\text{Dir}C}(E) = \{ r(e_1 - e_2) | e_1 \neq e_2 \in E, r \in \mathbb{Z}_p - \{ 0 \} \}$.  When working over $\mathbb{Z}_n$, $n$ not a prime, we will still use 
the same notation for these sets though some of them will no longer be cones.
\end{definition}

Another cone of importance for us later is the {\bf cone of balanced vectors} of dimension $mp$ over $\mathbb{Z}_p$. A vector is balanced if every element of 
the field $\mathbb{Z}_p$ occurs equally often as a coordinate of the vector. Clearly they exist only in dimensions which are multiples of $p$ 
and the set $B_m$ of balanced $mp$-dimensional vectors is easily seen to be a cone. In fact as adding any multiple of the all one vector to a balanced 
vector maintains balance, it is easy to see the set $B_m$ is a union of $2$-dimensional subspaces containing the line $L$ through the all one vector 
minus the line $L$ itself.

\section{Tilings}
\label{sect:tiling}

Let $E \subseteq \mathbb{Z}_n^d$. Then we say that $E$ tiles if it has a tiling partner $A$ such that the translates $\{ E+a | a \in A\}$ partition 
$\mathbb{Z}_n^d$. Note that in particular this implies $|E||A|=n^d$ and in particular $|E|$ is a divisor of $n^d$. When $n=p$ is prime this forces 
$|E|=p^r$ for some $0 \leq r \leq d$. The following are equivalent formulations of tiling sets:

\begin{theorem} 
\label{theorem:basicpicturetiling}
Let $E,A$ be subsets of $\mathbb{Z}_n^d$ with $|E||A|=n^d$. The following are equivalent:
\begin{itemize}
\item[(a)] Every $x \in \mathbb{Z}_n^d$ can be written uniquely in the form $e+a$, $e \in E, a \in A$. 
\item[(b)] $\mathbb{Z}_n^d = \cup_{a \in A} (E+a)$ where the union is disjoint.
\item[(c)] $E \star A = \mathfrak{1}$ where $\star$ is the discrete convolution operator, and we now use $E$ (respectively $A$) to stand for the characteristic function 
of the corresponding set. Here $\mathfrak{1}$ is the constant function with value one. 
\item[(d)] $\widehat{E}(m)\widehat{A}(m) = 0$ for all nonzero $m$. 
\item[(e)] $\text{Spt}(\widehat{E}) \cap \text{Spt}(\widehat{A}) = \emptyset$. 
\item[(f)] $Z(\widehat{E}) \cup Z(\widehat{A}) = \mathbb{Z}_n^d - \{ 0 \}$. 
\item[(g)] $\text{Dir}(E) \cap \text{Dir}(A) = \emptyset$. 

If $n=p$ is a prime, then these are also equivalent to: 
\item[(h)] $\text{\text{Dir}C}(E) \cap \text{DirC}(A) = \emptyset$. 
\end{itemize}
Furthermore if $s_1, s_2 \in \mathbb{Z}_p - \{ 0 \}$, $m_1, m_2 \in \mathbb{Z}_p^d$ then $(E,A)$ tiles if and only if 
$(s_1E+m_1, s_2A+m_2)$ tiles. Thus we can and often will assume $\vec{0}$ is in $E \cap A$ and indeed when we do this, 
we have to have $E \cap A = \{ \vec{0} \}$.
\end{theorem}
\begin{proof}
Condition (a) is immediately seen to be equivalent to the definition of a tiling set $E$ with tiling partner $A$. The definition we have taken as our initial one is (b).
(c) is seen to be equivalent immediately as $E \star A (x) = \sum_y E(y)A(x-y)$ counts the number of ways to write $x$ as a sum of something in $E$ with something in $A$. Taking the Fourier transform of the equation $E \star A = \mathfrak{1}$ yields the equation $\widehat{E}\widehat{A} = \frac{1}{n^d} \delta$ where 
$\delta$ is the Kronecker delta function. Evaluating this at $m=0$ gives $|E||A|=n^d$ which is supposed throughout this theorem anyway and so is redundant. 
Evaluating this at nonzero $m$ yields $\widehat{E}(m)\widehat{A}(m)=0$. This is then equivalent to (c) as the process is invertible using the inverse Fourier transform.
Conditions (e) and (f) are immediately seen to be equivalent to condition (d). Note that 
$\text{Dir}(E) \cap \text{Dir}(A) \neq \emptyset$ if and only if there exists $e \neq e' \in E, a \neq a' \in A$ such that $e-e'=a-a'$ if and only if 
$e+a = e'+a' = \alpha$ expresses $\alpha$ as a sum of an element of $E$ with an element of $A$ in two or more distinct ways.
Thus $(g)$ and $(a)$ are equivalent as condition $(g)$ happens if and only if $|E+A|=|E||A|=|\mathbb{Z}_n^d|$ if and only if condition (a) holds.
When $n=p$ is a prime and $s \in \mathbb{Z}_p - \{ 0 \}$ let $S = \{ x \in \mathbb{Z}_p^d | x = s e \text{ for some } e \in E \}$. We will write 
$S = s \cdot E$ in this situation. Then a quick computation shows that $\widehat{S}(\vec{m}) = \widehat{E}(s\vec{m})$ and so by Lemma~\ref{lemma:main} we conclude 
that $Z(\widehat{S})=Z(\widehat{E})$. As $|S|=|E|$ also, it follows by the equivalence of (d) and (a) that $(E,A)$ is a tiling pair if and only if $(sE,A)$ is a tiling pair. 
As it also clear that $(E,A)$ is a tiling pair if and only if $(A,E)$ is, it then follows that $(sE,tA)$ is a tiling pair whenever $(E,A)$ is where $s,t \in \mathbb{Z}_p-\{0\}$ are arbitrary. Thus if $(E,A)$ tiles, then $(sE, tA)$ tiles for all $t, s \in \mathbb{Z}_p - \{ 0 \}$ and thus by (g), $\text{Dir}(sE) \cap \text{Dir}(tA) = \emptyset$ and so 
$s \cdot \text{Dir}(E) \cap t \cdot \text{Dir}(A) = \emptyset$. As $\text{DirC}(E) = \cup_{s \in \mathbb{Z}_p - \{ 0 \}} s \cdot \text{Dir}(E)$ and similarly for $\text{DirC}(A)$ we conclude 
that $(E,A)$ a tiling pair implies $\text{DirC}(E) \cap \text{DirC}(A) = \emptyset$ which immediately gives the equivalence of $(g)$ and $(h)$ as we already knew the 
equivalence of $(g)$ with $(a)$.

Finally it is easy to check from the initial definition of a tiling pair that $(E,A)$ a tiling pair implies $(E+m_1, A + m_2)$ a tiling pair. Combined with the previous observation this implies $(sE + m_1, tA+m_2)$ is a tiling pair for all $m_1, m_2 \in \mathbb{Z}_p^d, s, t \in \mathbb{Z}_p - \{ 0 \}$. 
Finally note that $(E,A)$ tiling implies there is at most one element $x \in E \cap A$. This is because if there were two distinct elements $x, y \in E \cap A$ then 
$x + y = \alpha = y + x$ expresses $\alpha$ as a sum of something in $E$ and something in $A$ in two distinct ways contradicting the definition of tiling.
Thus if $\vec{0} \in E \cap A$ we have indeed $E \cap A = \{ \vec{0} \}$ as claimed.
\end{proof}

\begin{corollary}
\label{corr:afinetiling}
Let $(E,A)$ be a tiling pair in $\mathbb{Z}_p^d$ and let $\psi_{\mathbb{B},\vec{m}}: \mathbb{Z}_p^d \to \mathbb{Z}_p^d$ be a general affine transformation 
given by $\psi_{\mathbb{B},\vec{m}}(\vec{x}) = \mathbb{B} \vec{x} + \vec{m}$ for some invertible $d \times d$ matrix $\mathbb{B} \in GL_d(\mathbb{Z}_p)$, 
and translation vector $\vec{m} \in \mathbb{Z}_p^d$. Then: 
\begin{itemize}
\item[(1)] $(\psi_{\mathbb{B},\vec{m}}(E), \psi_{\mathbb{B}, \vec{m}}(A))$ is a tiling pair. Thus the property of being a tiling pair is invariant under affine transformations. 
\item[(2)] $(A,E)$ is a tiling pair. Thus the property of being a tiling pair is symmetric. 
\item[(3)] $(sE,tA)$ is a tiling pair for any $s,t \in \mathbb{Z}_p-\{0\}$. Thus the property of being a tiling pair is invariant under independent scalings. 
\end{itemize}
\end{corollary}
\begin{proof}
It is easy to check that $(E,A)$ a tiling pair then $(\mathbb{B}(E), \mathbb{B}(A))$ is a tiling pair directly from the definition of tiling pair. 
The rest then follows immediately from Theorem~\ref{theorem:basicpicturetiling}. 
\end{proof}

Note that part (3) of this corollary does not follow from part (1) unless $s=t$. The point of $(3)$ is that the scalings of $E$ and $A$ can be taken independently, 
i.e., by different amounts.

\section{Spectral sets}
A subset $E \subseteq \mathbb{Z}_n^d$ is a called a {\bf spectral set} if it has a spectrum $B \subseteq \mathbb{Z}_n^d$ such that the set of characters 
$$\{ \chi( b \cdot ( )) | b \in B \}$$ forms an {\bf orthogonal basis} of $L^2(E)$, the vector space of complex valued functions on $E$ with Hermitian inner product 
$<f, g> = \sum_{e \in E} f(e)\bar{g}(e)$. We call $(E,B)$ a spectral pair. Note that in this situation $L^2(E)$ has a basis of size $|B|$ by definition. On the other 
hand, $L^2(E)$ also has a basis of functions $\{ \delta_e | e \in E \}$ where $\delta_e(x)=1$ when $x=e$ and $\delta_e(x)=0$ when $x \neq e$. As any two basis of a finite dimensional vector space have the same size, we conclude $|E|=|B|$ when $(E,B)$ is a spectral pair.

\begin{theorem}
\label{theorem:basicpicturespectral}
Let $E,B$ be subsets of $\mathbb{Z}_n^d$ with $|E|=|B|$. The following are equivalent:
\begin{itemize}
\item[(a)] $(E,B)$ is a spectral pair. 
\item[(b)] Every function $f: E \to \mathbb{C}$ can be written as $f(x) = \sum_{b \in B} c_b \chi(b \cdot x)$ for unique complex numbers $\{ c_b | b \in B \}$ and all $x \in E$. 
Furthermore $\sum_{x \in E} \chi((b-b') \cdot x) = 0$ for distinct $b, b' \in B$. 
\item[(c)] $\widehat{E}(b-b')=0$ for all $b \neq b' \in B$.
\item[(d)] $\widehat{E}(\text{Dir}(B)) = 0$. 
\item[(e)] Write $E=\{e_1, \dots, e_N \}$ and $B=\{ b_1, \dots, b_N \}$. The complex matrix $\mathbb{M} \in Mat_{N}(\mathbb{C})$ whose $(i,j)$-entry is given by $\mathbb{M}_{ij} = \chi(e_i \cdot b_j)$ is a Butson-type 
Hadamard matrix, i.e., its entries are $n$th roots of unity and it satisfies $\mathbb{M}^*\mathbb{M}=N \mathbb{I}$, i.e., its rows (equivalently columns) are 
orthogonal and all have norm $N$ under the Hermitian inner product. Here, as usual, $\mathbb{M}^*$ is the complex conjugate of the transpose of $\mathbb{M}$.\\
In particular this means $(E,B)$ is a spectral pair if and only if $(B,E)$ is spectral pair. 
\end{itemize}
When $n=p$ a prime, these are also equivalent to: 
\begin{itemize}
\item[(f)] $\widehat{E}(\text{DirC}(B))=0$. 
\item[(g)] Write $E=\{e_1,\dots,e_N\}$ and $B=\{ b_1, \dots, b_N \}$. The $\mathbb{Z}_p$-valued matrix $\mathbb{L}=\log(\mathbb{M}) \in Mat_N(\mathbb{Z}_p)$ whose $(i,j)$-entry is given by $\mathbb{L}_{ij} = e_i \cdot b_j \in \mathbb{Z}_p$ is a $\mathbb{Z}_p$-valued Log-Hadamard matrix, i.e., the difference of any two distinct 
rows (equivalently columns) of $\mathbb{L}$ is a balanced vector. 
\end{itemize}
\end{theorem}
\begin{proof}
The first part of (b) merely states what it means for the set $\{ \chi(b \cdot ()) | b \in B \}$ to be a basis of $L^2(E)$. The equation 
$\sum_{x \in E} \chi((b-b')x) = 0$ for distinct $b, b' \in B$ is a restatement of the orthogonality of this basis. Thus $(b)$ is clearly equivalent to $(a)$. 
Note that given $|E|=|B|$, the fact that $\{ \chi(b \cdot ()) | b \in B \}$ is an orthogonal basis follows immediately from the orthogonality of these functions. 
This is because the functions $\chi(b \cdot ())$ are never zero and a collection of nonzero orthogonal elements must be linearly independent and hence form 
a basis as the number of elements in this collection is the same as the dimension of the ambient vector space by assumption. 
Now note that orthogonality $\sum_{x \in E} \chi((b-b') \cdot x) = 0$ is equivalent to $\widehat{E}(b-b')=0$ for all $b,b'$ distinct in $B$. Thus (c) is equivalent to (b).
(d) is equivalent to (c) by definition of $\text{Dir}(B)$. The orthogonality of the columns of $\mathbb{M}$ is equivalent to the orthogonality described in $(c)$ and 
so $(e)$ is equivalent to $(c)$ once one notes that the outputs of $\chi$ are always $n$th roots of unity. This column orthogonality is equivalent to the 
equation $\mathbb{M}^* \mathbb{M} = N \mathbb{I}$ which establishes $\frac{1}{N} \mathbb{M}^*$ is $M^{-1}$ and hence implies the equation 
$\mathbb{M}\mathbb{M}^* = N \mathbb{I}$ which then gives row orthogonality. (Thus in general the matrix $M$ will have orthogonal columns if and only if 
it has orthogonal rows). Thus we conclude that $(E,B)$ is a spectral pair if and only if $(B,E)$ is one.

When $n=p$ is a prime, Lemma~\ref{lemma:main} shows that $(f)$ is equivalent to $(d)$. It also shows that $\widehat{E}(b-b')=0$ if and only if the values 
of $e \cdot (b-b')$ equidistributes in $\mathbb{Z}_p$ as $e$ varies over $E$.  This happens if and only if the difference of any two distinct columns of 
$\mathbb{L}$ is a balanced vector. Note that as $(B,E)$ is a spectral pair also, we have $\widehat{B}(e-e')=0$ for distinct $e, e' \in E$ and so we can see the 
difference of two distinct rows of $\mathbb{L}$ is also a balanced vector.
\end{proof}

\begin{corollary}
\label{corr:spectralsize}
Let $E,B \subseteq \mathbb{Z}_n^d$ then:
\begin{itemize}
\item[(a)] When $|E|=|B|=1$, $(E,B)$ is a spectral pair so singleton sets are spectral. 
\item[(b)] When $|E|=|B|=n^d$ then $(E,B)$ is a spectral pair so the whole space $\mathbb{Z}_n^d$ is a spectral set. 
\item[(c)] If $E$ is a spectral set with $|E| > 1$ then $|E|$ is a multiple of $p$. 
\end{itemize}
\end{corollary}
\begin{proof}
Case (a) follows as $\chi(b \cdot x)$ is nonzero and hence forms a tautologically orthogonal basis of the $1$-dimensional vector space $L^2(E)$.
Case (b) follows by standard orthogonality of characters of $\mathbb{Z}_n^d$. \\
When $|E| > 1$ let $B$ be the spectrum of $E$ with $|E|=|B|=N$. Then by part (g) of Theorem~\ref{theorem:basicpicturespectral} we have the difference 
of distinct columns of the $N \times N$ dot-product  matrix $\mathbb{L}$ are balanced vectors. Thus if there are two distinct columns (i.e., $N > 1$), we must have 
$N$ is a multiple of $p$ as balanced vectors over $\mathbb{Z}_p$ must have dimension a multiple of $p$. 
\end{proof}

\begin{corollary}
\label{corr:spectral}
Let $E, B \subseteq \mathbb{Z}_n^d$. Suppose $(E,B)$ is a spectral pair then: 
\begin{itemize}
\item[(a)] The pair $(B,E)$ is also a spectral pair. 
\item[(b)] If $\mathbb{A}$ is an invertible $d \times d$ matrix over $\mathbb{Z}_n$ then $(\mathbb{A} \cdot E, \mathbb{A}^{-T} \cdot B)$ is also a spectral pair.
Here $\mathbb{A}^{-T}$ is the inverse transpose matrix of $\mathbb{A}$. 
\item[(c)] Let $n=p$ a prime then $(aE+m_1, bB+m_2)$ is a spectral pair for any $a,b \in \mathbb{Z}_p-\{ 0\}$ and $m_1, m_2 \in \mathbb{Z}_p^d$. 
\end{itemize}
\end{corollary}
\begin{proof}
Part (a) was proven during the proof of Theorem~\ref{theorem:basicpicturespectral}. To proof part (b), just note that the matrix $\mathbb{H}$ with $(i,j)$ entry 
given by $\chi(e_i \cdot b_j)$ with respect to some ordering of the sets $E, B$ is the same as the matrix with $(i,j)$ entry given by 
$\chi(\mathbb{A}e_i \cdot \mathbb{A}^{-T}b_j)$ Thus the pair $(E,B)$ is spectral if and only if the pair $(A \cdot E, A^{-T} \cdot B)$ is spectral by 
Theorem~\ref{theorem:basicpicturespectral} part (e). 

Now let $n=p$ a prime. If $(E,B)$ is a spectral pair then $\widehat{E}(\text{DirC}(B))=0$ by Theorem~\ref{theorem:basicpicturespectral} part (f). As $\text{DirC}(B)=\text{DirC}(b \cdot B + m_2)$ for any nonzero $b \in \mathbb{Z}_p$ and vector $m_2 \in \mathbb{Z}_p^d$, we conclude $\widehat{E}(\text{DirC}(b \cdot B + m_2)) = 0$ also and so $(E,b \cdot B + m_2)$ is a spectral pair. Using part (a) of this corollary a similar argument can be made for the first slot yielding that $(a \cdot E + m_1, b \cdot B + m_2)$ is a spectral pair for any $a,b \in \mathbb{Z}_p - \{ 0 \}$ and any vectors $m_1, m_2 \in \mathbb{Z}_p^d$.
\end{proof}

Thus just as for tiling sets, the last corollary shows that the property of being a spectral pair is symmetric and stable under independent scalings and translations. 
Furthermore the property of being a spectral set is an affine invariant, i.e., preserved under arbitrary invertible linear transformations and translations.

\subsection{Direction Cones and Projections}

Our final result in this section concerns a gap in the possible sizes of spectral sets $E \subseteq \mathbb{Z}_p^d$. It follows from a result on direction cones of sets in \cite{IMP11} - as the proof in that paper contains some minor typos we provide a self-contained cleaner proof here. For the proof, we will say that a linear map 
$\mathbb{A}: \mathbb{Z}_p^d \to \mathbb{Z}_p^d$ is a projection onto a subspace $V$ if $\mathbb{A}^2 = \mathbb{A}$ and $\mathbb{A} \vec{v} = \vec{v}$ if and only 
if $\vec{v} \in V$.

\begin{theorem}[From \cite{IMP11}] 
\label{theorem:directionconesizes}
Let $E \subseteq \mathbb{Z}_p^d$ then: 
\begin{itemize}
\item[(a)] If $|E| > p^{d-1}$ then $\text{DirC}(E)=\mathbb{Z}_p^d - \{ 0 \}$, i.e., $E$ determines all directions. 
\item[(b)] More generally if $|E| > p^{d-s}$ for some $1 \leq s \leq d$, then there is a (d-s+1)-dimensional subspace $V$ such that 
$\text{DirC}(E)$ projects onto $\text{DirC}(V)$. 
\end{itemize}
\end{theorem}
\begin{proof}
We first prove $(a)$ by contraposition. Suppose $E$ misses a direction. Then there is a line $\ell$ through the origin such that 
$\ell \cap \text{DirC}(E) = \emptyset$. Consider a linear projection $\mathbb{A}$ onto a hyperplane $H$ complementary to $\ell$, with kernel the line $\ell$. It follows that for distinct elements 
$e, e' \in E, e-e' \notin \ell$ and so $\mathbb{A}e \neq \mathbb{A}e' \in H$. Thus $\mathbb{A}$ bijects $E$ with a subset of $H$ and so 
$|E| \leq |H| = p^{d-1}$ and we are done.

Now let us prove part (b). Suppose we have constructed a chain of nested subspaces $\mathbb{Z}_p^d = V_0 \supseteq V_1 \supseteq V_2 \dots \supseteq V_k$
where $V_k$ has codimension $k$ in $\mathbb{Z}_p^d$ and where $\text{DirC}(E)$ does not project onto $\text{DirC}(V_j)$ under a projection to $V_j$ for 
$0 \leq j \leq k$ and furthermore $E' = \mathbb{A} E$ is a bijective image of $E$ under a projection onto $V_k$. Then by assumption $\text{DirC}(E')$ is not all of 
the (punctured) space $V_k$ and so misses a line $\ell$ in $V_k$. We can then project $V_k$ to a subspace $V_{k+1}$ of codimension $1$ in $V_k$ (and hence codimension $k+1$ in $\mathbb{Z}_p^d$) under a 
projection $B$ with kernel $\ell$. The map $B$ then takes $E'$ bijectively to a set $E'' \subseteq V_{k+1}$ as $\text{Dir}(E) \cap \ell = \emptyset$ and hence we have constructed an extension of the original chain.

This process must eventually terminate within $d$ steps and so there must exist a first point in this chain  
say $V_m$ where $\text{DirC}(E)$ maps onto the (punctured taking the origin out) space $V_m$. As $E$ bijects with a subspace of $V_m$ by construction, we have 
$p^{d-s} < |E| \leq p^{d-m}$ and so $d-m > d-s$ and so $d-m \geq d-s+1$. Taking a (d-s+1) dimensional subspace of $V_m$ as $V$ yields a $(d-s+1)$-dimensional subspace such that $\text{DirC}(E)$ projects onto $\text{DirC}(V)=V-\{ 0 \}$ and hence gives (b).

\end{proof}

This yields a corollary which says that there are no spectral sets of sizes strictly between $p^{d-1}$ and $p^d$ in $\mathbb{Z}_p^d$.

\begin{corollary}
\label{corollary:spectralgaphighend}
Let $(E,A) \subseteq \mathbb{Z}_p^d$ be a spectral pair. Then $|E| < p^d$ implies $|E| \leq p^{d-1}$.
\end{corollary}
\begin{proof}
Suppose $|E| < p^d$ i.e., $E$ is not the whole space.
By Theorem~\ref{theorem:basicpicturespectral}, we have $\widehat{E}(\text{DirC}(A))=0$.
Thus by Theorem~\ref{theorem:directionconesizes}, we conclude that if $|A| > p^{d-1}$ then $\text{DirC}(A)=\mathbb{Z}_p^d - \{ 0 \}$ and hence 
$\widehat{E}(m)=0$ for all nonzero $m$. This implies the characteristic function of $E$ is a constant function so $E$ is either the empty set or the whole space 
both contrary to our hypothesis. Thus $|A|=|E| \leq p^{d-1}$ and we are done.

\end{proof}

We warn the reader that in Theorem~\ref{theorem:directionconesizes}, one cannot replace the condition that $\text{DirC}(E)$ projects onto $\text{DirC}(V)$ for a subspace 
$V$ with the condition that it contains $\text{DirC}(V)$ as the two concepts are quite different in general as the following example shows:

\begin{example}[Random Cones]
Let $p$ be a prime and let us generate random cones in $\mathbb{Z}_p^d$ with the following process: Fix a real number $0 < \alpha < 1$ and 
let $C$ be the random cone obtained by letting each line through the origin be in $C$ with probability $\alpha$ independently of whether other lines are in $C$ or not.
Note that any given cone can be the outcome of this random binomial process though cones where the proportion of lines in $C$ is close to $\alpha$ are most likely. 

Then if $d \geq 3$, almost surely as $p \to \infty$, a random cone $C$ has the property that it linearly projects {\bf onto} $H - \{ \vec{0} \}$ for every hyperplane $H$ and hence onto the direction cone 
of any proper subspace of $\mathbb{Z}_p^d$ {\bf and } also has the property that $C$ almost surely does not contain $V - \{ 0 \}$ for any subspace $V$ of dimension $\geq 2$.
\end{example}
\begin{proof}
Given a subspace $V$ of dimension two, $V-\{0 \}$ consists of the union of $\frac{p^{2}-1}{p-1}=p+1$ disjoint (punctured by taking origin out) lines through the origin.
The chance that the random cone $C$ contains $V-\{ 0 \}$ is hence $\alpha^{p+1}$. As the number of $2$-dimensional subspaces in $\mathbb{Z}_p^d$ when 
$d \geq 3$ is $\frac{(p^d-1)(p^d-p)}{(p^2-1)(p^2-p)}$ we see that an upper bound on the probability that the random cone $C$ contains some (punctured) 2-dimensional subspace is $\frac{(p^d-1)(p^d-p)}{(p^2-1)(p^2-p)}\alpha^{p+1}$ which tends to zero as $p \to \infty$ (and $d$ fixed). Thus almost surely as 
$p \to \infty$, a random cone contains no (punctured) subspace of dimension $\geq 2$ as claimed.

Now let $H$ be a hyperplane through the origin and $\vec{v} \in H - \{ 0 \}$. Let $\pi: \mathbb{Z}_p^d \to H$ be a linear projection onto the hyperplane $H$.
Now $\pi^{-1}(v)$ is a line which does not go through the origin and hence intersects exactly $p$ lines through the origin. Thus the probability that 
$\pi(C)$ does not contain $v$ is equal to $(1-\alpha)^p$. It follows that the probability that $\pi(C) \neq H-\{ 0\}$ is bounded above by $(p^{d-1}-1)(1-\alpha)^p$. 
As there are $\frac{p^d-1}{p-1}$ hyperplanes in $\mathbb{Z}_p^d$, the probability that there is some hyperplane $H$ such that $C$ does not project onto $H-\{0\}$ under some linear projection is bounded above by $\frac{p^d-1}{p-1} (p^{d-1}-1)(1-\alpha)^p$ which tends to zero as $p \to \infty$. Thus with probability $1$ 
as $p \to \infty$, a random cone will linearly project {\bf onto every} hyperplane. Combined with the previous paragraph, this finishes the proof that the example works as claimed.

\end{proof}

\subsection{Extracting spectral pairs in a given dimension from log-Hadamard matrices of suitable rank}

In this subsection, let $n=p$ a prime. We will prove an equivalence between the existence of a spectral pair $(E,B)$ in $\mathbb{Z}_p^d$ with $|E|=|B|=m$ and 
the existence of a $m \times m$ log-Hadamard matrix over $\mathbb{Z}_p$ of rank less than or equal to $d$. 

\begin{theorem}
\label{theorem:logHadamardSpectralEquivalence}
Let $p$ be a prime. Then the existence of a spectral pair $(E,B)$ in $\mathbb{Z}_p^d$ with $|E|=|B|=m$ is equivalent to 
the existence of a $m \times m$ log-Hadamard matrix with entries in $\mathbb{Z}_p$ with rank less than or equal to $d$.
\end{theorem}
\begin{proof}
If $E, B \subseteq \mathbb{Z}_p^d$ with $|E|=|B|=m$, let us order the sets $E=\{ e_1, \dots, e_m \}$, $B=\{ b_1, \dots, b_m \}$ arbitrarily. Then we can associate matrices $\mathbb{E}$ and $\mathbb{B}$ to these sets where the $i$th row of $\mathbb{E}$ (respectively $\mathbb{B}$) is $e_i$ (respectively $b_i$).  Thus 
$\mathbb{E}$ and $\mathbb{B}$ are $m \times d$ matrices and hence have rank less than or equal to $d$. 
Let $\mathbb{L} = \mathbb{E} \mathbb{B}^T$ be the corresponding matrix of dot products of vectors of $E$ and $B$ which will be a $m \times m$ 
matrix with entries in $\mathbb{Z}_p$.  It has rank less than or equal to $d$ as the rank of a product of matrices is no more than the minimum of the ranks of the factors.
Finally if $(E,B)$ is a spectral pair, by Theorem~\ref{theorem:basicpicturespectral}, $\mathbb{L}$ is a log-Hadamard matrix. 
Thus the existence of a spectral pair $(E,B)$ in $\mathbb{Z}_p^d$ with $|E|=|B|=m$ has been shown to imply the existence of a 
$m \times m$ log-Hadamard matrix $\mathbb{L}$ with entries in $\mathbb{Z}_p$ and rank less than or equal to $d$. 

It remains to prove the converse so assume that we are given a $m \times m$ log-Hadamard matrix $\mathbb{L}$ of rank less than or equal to $d$.
Using the standard basis of $\mathbb{Z}_p^m, \mathbb{L}$ gives rise to a linear operator $\mathbb{L}: \mathbb{Z}_p^m \to \mathbb{Z}_p^m$ 
whose image is at most $d$ dimensional. This operator can be factored as the composition of two linear operators 
$\mathbb{L}: \mathbb{Z}_p^m \to Im(L)$ and $J: Im(L) \to \mathbb{Z}_p^m$ where $J$ is just an inclusion. Choosing a basis for $Im(L)$, these operators can 
be represented by $k \times m$ matrix $\mathbb{C}$ and $m \times k$ matrix $\mathbb{E}$ with $k=\dim(Im(L)) \leq d$. 
Thus $\mathbb{L}=\mathbb{E} \mathbb{C}$. Setting $\mathbb{B}=\mathbb{C}^T$ we get $\mathbb{L} = \mathbb{E} \cdot \mathbb{B}^T$ for 
$m \times k$ matrices $\mathbb{E}$ and $\mathbb{B}$. The row vectors of $\mathbb{E}$ must be distinct as if two rows of $\mathbb{E}$ are the same 
then the corresponding rows of $\mathbb{L}=\mathbb{E} \mathbb{B}^T$ are the same contradicting that $\mathbb{L}$ is log-Hadamard (the difference of distinct rows must be balanced and the zero vector is not balanced). Similarly distinct rows of $\mathbb{B}$ cannot be the same or two distinct columns of $\mathbb{L}$ would equal which would contradict $\mathbb{L}$ being log-Hadamard. Thus we may define sets $\mathbb{E}, \mathbb{B} \subseteq \mathbb{Z}_p^k$ 
with $|E|=|B|=m$ by letting $E$ be the set containing the row vectors of $\mathbb{E}$ and $B$ be the set containing the row vectors of $\mathbb{B}$. 
As $\mathbb{L}=\mathbb{E}\mathbb{B}^T$ is log-Hadamard, Theorem~\ref{theorem:basicpicturespectral} implies that $(E,B)$ is a spectral pair in 
$\mathbb{Z}_p^k$. Finally as $k \leq d$ we may view $\mathbb{Z}_p^k \subseteq \mathbb{Z}_p^d$ via an inclusion whose image are the vectors with 
zero in the last $d-k$ slots. This identifies $(E,B)$ as a spectral pair in $\mathbb{Z}_p^d$ (the pair is still spectral as its dot product matrix is unchanged) 
with $|E|=|B|=m$ and the proof is hence complete.

\end{proof}

An explicit way to construct the spectral pair $(E,B)$ in $\bb{Z}_p^d$ with $|E|=|B|=m$ given an $m \times m$ log-Hadamard matrix $\bb{L}$ with entries in $\bb{Z}_p$ of rank equal to $d$ is as follows.

Let $r_1,\dots,r_d$ be $d$ linearly independent rows of $\bb{L}$ (each $r_i \in \bb{Z}_p^m$). Every row of $\bb{L}$ is a linear combination of the $r_j$ so there exist $e_{ij} \in \bb{Z}_p$ such that for all $i=1,\dots,m$
\[
(\text{$i$th row of } \bb{L}) = \sum_{j=1}^d e_{ij} r_j .
\] 

Now let $\bb{E}$ be the $m \times d$ matrix whose $(i,j)$-entry is $e_{ij}$ and $\bb{B}$ be the $m \times d$ matrix whose $i$th column (and not row) is $r_i$. The choice of the $e_{ij}$ ensures that $\bb{L} = \bb{E} \bb{B}^T$.

Finally, $E$ is taken to consist of the $m$ distinct rows of $\bb{E}$ and $B$ to consist of the $m$ distinct rows of $\bb{B}$.

\section{More on log-Hadamard matrices and balanced vectors}

We present some results on log-Hadamard matrices and balanced vectors that are of a more theoretical nature. 

\subsection{Equivalence of log-Hadamard matrices}

Fix $p$ a prime. Given a balanced $\mathbb{Z}_p$-vector $\vec{v}$, adding a multiple of the all $1$ vector $\mathfrak{1}$ to $\vec{v}$ preserves the balanced property. Thus given a log-Hadamard matrix $\mathbb{L}$, we may add (different) multiples of $\mathfrak{1}$ to the rows (and columns) of $\mathbb{L}$ to get a 
new log-Hadamard matrix $\mathbb{L}$. We may also permute rows or columns and the log-Hadamard property is preserved.

\begin{definition}
We say that two $m \times m$ log-Hadamard matrices $\mathbb{L}, \mathbb{L}'$ over $\mathbb{Z}_p$ are equivalent if one can be obtained from the other through a sequence of row permutations, column permutations and additions of multiples of $\mathfrak{1}$ to rows (and columns).
\end{definition}

Thus every log-Hadamard matrix is equivalent to a {\bf dephased} log-Hadamard matrix, i.e., one whose leftmost column and topmost row consists of all zeros.
Note in a dephased log-Hadamard matrix, all rows besides the top row consist of balanced vectors and all columns besides the leftmost one consist of balanced vectors. The reader can check that this notion of equivalence is equivalent to the usual notion for Hadamard matrices in the literature.

Note given a spectral pair $(E,B) \in \mathbb{Z}_p^d$, we may translate $E$ and $B$ separately to get a new spectral pair 
$(E',B') \in \mathbb{Z}_p^d$ with $\vec{0} \in E \cap B$. Ordering these sets so that $\vec{0}$ is listed first, the corresponding dot product matrix 
$\mathbb{L} = \mathbb{E} \mathbb{B}^T$ is a dephased log-Hadamard matrix. 

Thus we get this immediate corollary of Theorem~\ref{theorem:logHadamardSpectralEquivalence}:

\begin{corollary}
\label{corollary:dephasedrank}
Fix $p$ a prime. The following are equivalent:
\begin{itemize}
\item[(1)] There exists a $m \times m$ log-Hadamard matrix $\mathbb{L}$ of rank less than or equal to $d$ over $\mathbb{Z}_p$. 
\item[(2)] There exists a spectral pair $(E,B)$ in $\mathbb{Z}_p^d$ with $|E|=|B|=m$. 
\item[(3)] There exists a spectral pair $(E,B)$ in $\mathbb{Z}_p^d$ with $|E|=|B|=m$ and $\vec{0} \in E \cap B$. 
\item[(4)] There exists a $m \times m$ {\bf dephased} log-Hadamard matrix $\mathbb{L}$ of rank less than or equal to $d$ over $\mathbb{Z}_p$. 
\end{itemize}
\end{corollary}

Note $1 \times 1$ matrices are trivially log-Hadamard but the next smallest possible size is $p \times p$ as the difference of distinct rows must be balanced and 
hence in particular have dimension a multiple of $p$. A $p \times p$ log-Hadamard matrix is equivalent to a {\bf unique} matrix of the form 
$$
\mathbb{L} = \begin{bmatrix} 0 & 0 & 0 & \dots & 0 \\ 0 & 1 & 2 & \dots & p-1 \\ 0 & 2 & * & \dots  & * \\ \vdots & \vdots & \vdots & \vdots & \vdots \\
0 & p-1 & * & \dots & * \end{bmatrix}
$$
i.e. with zero top row and left column and where the next row and column is given by the balanced vector $(0,1,2,\dots,p-1)$ (or its transpose). 
We shall call this the {\bf special dephased} form of a $p \times p$ log-Hadamard matrix.

It is convenient to use the convention that the leftmost column and top row are called the $0$th column and $0$th row instead of the usual convention of calling it the $1$st. With this convention the $p \times p$ matrix $[ij]$ whose $(i,j)$-entry is $ij$, $0 \leq i,j \leq p-1$ is a special dephased log-Hadamard matrix. 

It is not hard to show that $[ij]$ is the only special dephased log-Hadamard matrix when $p=3,5$ but for higher $p$ there exist other examples. Note that in a special dephased $p \times p$ log-Hadamard matrix, all rows past the $0$th and $1$st row correspond to bijections $\psi: \mathbb{Z}_p \to \mathbb{Z}_p$ such that 
$\psi(0)=0$ and $x \to \psi(x) - x$ is also a bijection of $\mathbb{Z}_p$. Affine bijections of the form $\Psi(j)=ij$ for a constant $i \in \mathbb{Z}_p - \{ 0, 1 \}$ are examples of such and correspond to the matrix $[ij]$. For $p=3,5$ these are the only such bijections but for $p \geq 7$ there exist other examples as the reader can computationally verify.

In general an $mp \times mp$ log-Hadamard matrix will be said to be in special dephased form if it is in the form where its $0$th row and column are the zero vector 
and its $1$st row and column are the vector $(0,1,2, \dots, p-1, 0,1,2 \dots, p-1, \dots, 0, 1, 2, \dots, p-1)$ (or its transpose). Any $mp \times mp$ log-Hadamard matrix is equivalent to such a matrix.

It is useful to consider the structure of the $2$nd row in such a special dephased log-Hadamard matrix. We introduce variables to codify the choices for this row.

Let $x_{ij}$ be the number of $i$'s in the 2nd row below a $j$ in the $1$st row where $0 \leq i,j \leq p-1$. We put these in a $p \times p$ matrix called  
$\mathbb{X}$ whose $(i,j)$-entry is $x_{ij}$. The reader may verify that as the $2$nd row is balanced, the column sums of this matrix are all $m$ 
and that as the $1$st row is balanced, the row sums of this matrix are all $m$ also. Finally as the difference of the $2$nd row and $1$st row is balanced, 
all the ``diagonal'' sums $\sum_{i=0}^{p-1} x_{i,i+s}$ are also equal to $m$. The reader may also verify that given any matrix with nonnegative integer entries, 
with these constraints one can construct a suitable $2$nd row (unique up to permutations amongst slots in the $1$st row with the same entry).
Similar comments apply to any row after the 2nd one also.

Thus in the construction of special dephased $mp \times mp$ log-Hadamard matrices, we are lead to an important class of matrices which we define now:

\begin{definition}
Fix a prime $p$. A $p \times p$, matrix $\mathbb{X}$ with nonnegative integer entries whose row sums, column sums and diagonal sums $\sum_{i} \mathbb{X}_{i,i+s}$ for any $0 \leq s \leq p-1$ all add up to $m$ is called a Davey matrix of weight $m$.  (Here we take the row and column indices of $\mathbb{X}$ to take values in $\mathbb{Z}_p$.) The set of such matrices will be denoted by $D_m$. The sum of a matrix in $D_m$ with one in $D_n$ is in $D_{m+n}$ so 
$D = \cup_{m=0}^{\infty} D_m$ is an Abelian monoid under matrix addition called the Davey monoid. The Grothendieck construction on $D$ yields an Abelian group 
called $\widehat{D}$, the Davey group. To emphasize the dependence on the prime $p$, we sometimes will write $D(p)$ and $\widehat{D}(p)$.
\end{definition}

We will carefully go through some examples to clarify the structure of the Davey monoid $D(2)$ and $D(3)$.

It is a fact from the matching theory of graphs (an exercise in \cite{W01}) that a matrix with nonnegative integer entries whose row sums and column sums all equal to $m$ can be written as a sum of $m$ permutation matrices (not necessarily distinct).

When $p=2$ we are dealing with $2 \times 2$ matrices and there are only two permutation matrices. Thus a Davey matrix will be a matrix of the form 
$\mathbb{X} = \ell_1 \mathbb{I} + \ell_2 \mathbb{T}$ where 
$$
\mathbb{T}=\begin{bmatrix} 0 & 1 \\ 1 & 0 \end{bmatrix}.
$$
Thus 
$$
\mathbb{X} = \begin{bmatrix} \ell_1 & \ell_2 \\ \ell_2 & \ell_1 \end{bmatrix}
$$
and the row sums and column sums all equal to $m=\ell_1 + \ell_2$ which is hence the weight. 

 in order to be a Davey matrix, all the diagonal sums (not just the main diagonal) have to also equal $m = \ell_1 + \ell_2$. 
This forces $\ell_1 + \ell_1 = \ell_1 + \ell_2 = \ell_2 + \ell_2$ and so $\ell_1 = \ell_2$ and so the weight $m=2\ell_1$ has to be even. 
We have proven
\begin{theorem}
\label{theorem:Davey2}
Any $2 \times 2$ Davey matrix is of the form $\ell \begin{bmatrix} 1 & 1 \\ 1 & 1 \end{bmatrix}$ and so has even weight $m=2\ell$. Thus the Davey monoid $D(2)$ is isomorphic to the natural numbers under addition via isomorphism $\theta: \mathbb{N} \to D(2)$ given by $\theta(\ell)=\ell \begin{bmatrix} 1 & 1 \\ 1 & 1 \end{bmatrix}$. 
Thus the Davey group $\widehat{D}(2)$ is isomorphic to the integers, i.e., $\widehat{D}(2) \cong \mathbb{Z}$.
\end{theorem}

This yields the following corollary:

\begin{corollary}
\label{corollary:multipleof4}
Let $(E,B)$ be a spectral pair in $\mathbb{Z}_2^d$, then either $|E|=1, 2$ or $|E|$ is a multiple of $4$.
\end{corollary}
\begin{proof}
We may assume $\vec{0} \in E \cap B$ and $|E| > 2$. We then know that 
$|E|=2m$ is a multiple of $2$ by Theorem~\ref{theorem:basicpicturespectral}. The matrix $\mathbb{L}=\mathbb{E}\mathbb{B}^T$ is a dephased 
$2m \times 2m$ log-Hadamard matrix which can be made special dephased by suitable ordering of the elements of $E$ and $B$. As 
$|E| > 2$ we know there is a $2$nd row (labelling top row as $0$th row) and so by the previous discussion there must exist a 
$2 \times 2$ Davey matrix of weight $m$ whose $(i,j)$-entry is given by the number of $i$'s in the 2nd row below a $j$ in the 1st row. By Theorem~\ref{theorem:Davey2} we conclude $m=2k$ is even and so $|E|=2m=4k$ is a multiple of $4$.
\end{proof}

The Davey monoid and group often encode intricate structure constraints for log-Hadamard matrices and their corresponding Butson-type Hadamard matrices. 
However the difficulty of explicit computation of these goes up enormously with the prime $p$. 

We will now conduct the computation for the prime $p=3$. 
Thus we are considering $3 \times 3$ matrices with nonnegative integer entries $\mathbb{X}$ whose column sums, row sums and (all) diagonal sums add up to a fixed weight $m$. Such a matrix must in particular be a sum of $m$ (not necessarily) distinct permutation matrices. There are $3!=6$ of these now so the computation is more ugly:
$$
\mathbb{X} = \ell_1 \mathbb{I} + \ell_2 \sigma_1 + \ell_3 \sigma_2 + \ell_4 \sigma_3 + \ell_5 \tau + \ell_6 \tau^2
$$
where 
$\sigma_1 = \begin{bmatrix} 0 & 1 & 0 \\ 1 & 0 & 0 \\ 0 & 0 & 1 \end{bmatrix}, \sigma_2 = \begin{bmatrix} 0 & 0 & 1 \\ 0 & 1 & 0 \\ 1 & 0 & 0 \end{bmatrix}, 
\sigma_3 = \begin{bmatrix} 1 & 0 & 0 \\ 0 & 0 & 1 \\ 0 & 1 & 0 \end{bmatrix}, \tau= \begin{bmatrix} 0 & 1 & 0 \\ 0 & 0 & 1 \\ 1 & 0 & 0 \end{bmatrix}
$ and $\tau^2= \begin{bmatrix} 0 & 0 & 1\\ 1 & 0 & 0 \\ 0 & 1 & 0 \end{bmatrix}$.

The matrices $\sigma_1, \sigma_2, \sigma_3$ have all their diagonal sums (not just main diagonal) equal and so 
$\mathbb{X} = \ell_1 \mathbb{I} + \ell_2 \sigma_1 + \ell_3 \sigma_2 + \ell_4 \sigma_3 + \ell_5 \tau + \ell_6 \tau^2$ is Davey if and only if 
$\ell_1 \mathbb{I} + \ell_5 \tau + \ell_6 \tau^2$ is Davey. A quick computation of the diagonal sums of this last expression shows that this happens if and only if 
$3 \ell_1 = 3 \ell_5  = 3 \ell_6$ in which case this last expression equals $\ell_1 \begin{bmatrix} 1 & 1 & 1 \\ 1 & 1 & 1 \\ 1 & 1 & 1 \end{bmatrix}$ the all one-matrix.
However this all one matrix is also the sum $\sigma_1 + \sigma_2 + \sigma_3$. Thus we conclude that any $3 \times 3$ Davey matrix can be written 
uniquely in the form $s_1 \sigma_1 + s_2 \sigma_2 + s_3 \sigma_3$.

We summarize these results in the following theorem:

\begin{theorem}
\label{theorem:Davey3}
The set of $3 \times 3$ Davey matrices are those matrices of the form 
$$s_1 \begin{bmatrix} 0 & 1 & 0 \\ 1 & 0 & 0 \\ 0 & 0 & 1 \end{bmatrix} + s_2 \begin{bmatrix} 0 & 0 & 1 \\ 0 & 1 & 0 \\ 1 & 0 & 0 \end{bmatrix} + s_3 \begin{bmatrix} 1 & 0 & 0 \\ 0 & 0 & 1 \\ 0 & 1 & 0 \end{bmatrix}$$ for suitable unique nonnegative integers $s_1, s_2, s_3$. Thus $D(3) \cong \mathbb{N}^3$ and 
$\widehat{D}(3) \cong \mathbb{Z}^3$.
\end{theorem}

We will use Theorem~\ref{theorem:Davey3} to give a self-contained human-readable proof that the Fuglede conjecture holds in $\mathbb{Z}_3^3$ in a later 
section. Of course computer calculations can also be used to establish the same result. The precise form needed to prove this conjecture is stated in the next 
corollary:

\begin{corollary}
\label{corollary:Davey3}
Given $\vec{x}, \vec{y}$ equi-dimensional balanced vectors over $\mathbb{Z}_3$ whose difference is $\vec{x}-\vec{y}$ is also balanced, we have 
the {\bf triplet rule}: For any fixed distinct values for $i,j,k$ in $\mathbb{Z}_3$, the number of coordinates where the value of $\vec{x}$ is $i$ and 
the value of $\vec{y}$ is $j$ is equal to the number of coordinates where the value of $\vec{x}$ is $j$ and the value of $\vec{y}$ is $i$ and these also equal 
the number of coordinates where both $\vec{x}$ and $\vec{y}$ have value $k$.
\end{corollary}
\begin{proof}
Let $\bb{X}$ be the matrix whose $(m,n)$-entry is equal to the number of coordinates where $\vec{x}$ has value $m$ and $\vec{y}$ has value $n$.
$\bb{X}$ is a $3 \times 3$ Davey matrix. By theorem~\ref{theorem:Davey3} we must have 
$$
\bb{X} = \begin{bmatrix} s_3 & s_1 & s_2 \\ s_1 & s_2 & s_3 \\ s_2 & s_3 & s_1 \end{bmatrix}
$$
for some nonnegative integers $s_1, s_2, s_3$. The result follows immediately from this.
\end{proof}

\subsection{Balanced vectors}

Fix a prime $p$. Recall that a $\mathbb{Z}_p$ vector is balanced if each element of $\mathbb{Z}_p$ occurs the same number of times as a coordinate of the vector. 
It is clear that balanced vectors exist only in dimensions $mp$ where $m \geq 1$ is an integer. We denote by $B_m$ the set of balanced 
$mp$-dimensional vectors over $\mathbb{Z}_p$. 

Notice that $B_m$ is a cone as if $\vec{v} \in B_m$, then $c\vec{v} \in B_m$ for all $c \in \mathbb{Z}_p-\{ 0 \}$. Furthermore as adding any multiple 
of the all one vector $\mathfrak{1}$ to a balanced vector $\vec{v}$ yields a balanced vector, we see that $B_m$ is a union of 
$2$-planes through the origin containing the line through $\mathfrak{1}$ with this line taken out. 
Indeed if $\vec{v} \in B_m$ then $c\vec{v} + m \mathfrak{1} \in B_m$ for all $m \in \mathbb{Z}_p, c \in \mathbb{Z}_p - \{ 0 \}$. Thus the whole $2$-plane spanned by 
$\vec{v}$ and $\mathfrak{1}$ minus the line through $\mathfrak{1}$ is contained in $B_m$.

The symmetric group $\sigma_{mp}$ acts on $\mathbb{Z}_p^{mp}$ by permuting coordinates and this action preserves the cone $B_m$. As every subset of a vector space over a finite field is an algebraic set (zero set of collection of polynomials) we conclude that $B_m$ must be the common zero set of a collection of symmetric polynomials which we will identify explicitly next.

Let $x_1, x_2, \dots, x_{mp}$ denote the standard coordinate variables of $\mathbb{Z}_p^{pm}$.  Consider the formal polynomial in new variable $t$ given by 
$p(t)=(t-x_1)(t-x_2) \dots (t-x_{mp})$ then $p$ is a monic polynomial in the variable $t$ of degree $mp$.

The vector $\vec{x}=(x_1,\dots, x_{mp})$ is balanced if and only if the associated formal polynomial $p(t)=(t-x_1)(t-x_2) \dots (t-x_{mp})$ considered in 
$\mathbb{Z}_p[t]$ is equal to the polynomial $$t^m(t-1)^m(t-2)^m \dots (t-(p-1))^m = (t(t-1)(t-2)\dots (t-(p-1)))^m = (t^p-t)^m.$$ 

Recall the elementary symmetric polynomials $\sigma_1, \dots, \sigma_{mp}$ given by $\sigma_1 = x_1 + \dots + x_{mp}, \sigma_2 = \sum_{i < j} x_i x_j, 
\sigma_3 = \sum_{i < j < k} x_ix_jx_k, \dots, \sigma_{mp} = x_1 x_2 \dots x_{mp}$ satisfy 
$p(t)=(t-x_1)(t-x_2) \dots (t-x_{mp}) = \sum_{j =0}^{mp} (-1)^j t^{mp-j} \sigma_j$.  (Here we adopt the convention that $\sigma_0=1$ for convenience.)

Thus the we see that $\vec{x}=(x_1,\dots, x_{mp})$ is a balanced vector if and only if $p(t)=\sum_{j=0}^{mp} (-1)^j t^{mp-j} \sigma_j = (t^p-t)^m$.
Expanding $(t^p-t)^m = \sum_{i=0}^m \binom{m}{i} (-1)^{i} t^{p(m-i)+ i}$ using the binomial formula, we see this is equivalent to the conditions 
that $\sigma_j = 0$ when $j$ is not a multiple of $p-1$ and $\sigma_{i(p-1)} = (-1)^i \binom{m}{i}$ for all $0 \leq i \leq m$ and $\sigma_{i(p-1)}=0$ for all $i > m$.

We have thus proven the following theorem:

\begin{theorem}[Geometry of the set of balanced vectors]
Fix $p$ a prime and let $B_m$ denote the set of balanced vectors in $\mathbb{Z}_p^{pm}$. Then
$B_m$ is a cone. In fact $B_m$ is a union of $2$-spaces containing the line through $\mathfrak{1}$, the all one vector, minus this line. 
Furthermore $B_m$ can be given as the solution set to the following polynomial equalities: \\
$$\sigma_j=0$$ when $j$ is not a multiple of $p-1$, and $$\sigma_{i(p-1)}=(-1)^i \binom{m}{i}$$ where $\sigma_i$ are the elementary symmetric polynomials in 
the coordinate variables of $\mathbb{Z}_p^{pm}$.
\end{theorem}

Note the elementary symmetric polynomial $\sigma_j$ is homogeneous of degree $j$ in the sense that 
$\sigma_j(cx_1, \dots, cx_{mp}) = c^j \sigma_j(x_1, \dots, x_{mp})$. The fact that $B_m$ is a cone and hence invariant under nonzero scalings is consistent 
with the equations obtained above as $c^{p-1}=1$ when $c \in \mathbb{Z}_p - \{ 0 \}$.

The following is a cute corollary:

\begin{corollary}[Wilson's Theorem]
As the vector $\vec{x}=(0,1,2,\dots,p-1) \in B_1$ it satisfies $\sigma_{p-1}(0,1,2,\dots,p-1)=-1$. This simplifies to $(p-1)! = -1$ in $\mathbb{Z}_p$ 
which is the well known Wilson's Theorem.
\end{corollary}

\section{Preliminary Results}

We proceed to prove parts of Theorem \ref{theorem:main1} on tiling and spectral sets of specified size. 

\subsection{Graphs of functions}

We begin with a short account of graphs. Let $V$ be a $k$-dimensional subspace of $\bb{Z}_p^d$ and $W$ any complementary subspace in the sense that $\bb{Z}_p^d$ is the direct sum of $V$ and $W$. To keep the notation simple we sometimes express elements of $\bb{Z}_p^d$ as the ordered pairs $(v,w)$ for unique $v \in V$ and $w \in W$.

Given a  function $f: V \to W$ we define the graph of $f$ (relative to $V$, $W$) to be the following subset of $\bb{Z}_p^d$:
\[
\text{Graph}_{V,W}(f) = \{ (v,f(v)) : v \in V\} \subseteq \mathbb{Z}_p^d. 
\]
We suppress the subscripts in the cases where $V,W$ are unambiguously determined from the context.

Note that upon a choice of basis for $V$ and $W$, the set of functions $f: V \to W$ can be identified with the set of functions 
$f: \mathbb{Z}_p^s \to \mathbb{Z}_p^{d-s}$ when $\dim(V)=s$. 

Note in the case $\dim(V)=0$ the graph sets $\text{Graph}_{V,W}(f)$ are just the singleton subsets and in the case $\dim{V}=d$, the graph sets are the whole space.

The purpose of this subsection is to prove the following characterisation of graphs:
\begin{proposition}
\label{prop:graph=tile}
Let $V$ be a subspace of $\bb{Z}_p^d$, $W$ any complementary subspace of $V$, and $E \subseteq \bb{Z}_p^d$. $(E,V)$ is a tiling pair if and only if $E$ is a graph of a function $f: W \to V$.
\end{proposition}
\begin{proof}
Suppose first that $(E,V)$ is a tilling pair. By Corollary \ref{corr:afinetiling} we know that $(V,E)$ is a tiling pair. We deduce from this that $E$ contains precisely one element in each coset of $V$. Suppose that both $e$ and $e'$ belong to the coset $x+V$. Then both cosets $e+V$ and $e'+V$ equal $x+V$ ($V$ is a subspace). As $(V,E)$ tiles, we have $e=e'$.

Therefore one can identify elements of $E$ with cosets of $V$. The latter can also be identified with elements of the complementary subspace $W$. So to each $w \in W$ corresponds a unique $e_w\in E$ given by $E \cap (w+V) = \{e_w\}$. In the notation introduced above $e_w = (w, f(w))$ where $f(w)$ is a unique vector in $V$. This defines a function 
$f: W \to V$ with the property that $E=\text{Graph}_{W,V}(f)$. Note that we made use of the fact that $V$ is a complementary subspace to $W$.
 
Conversely, suppose that $f: W \to V$ is a function and $E=\{(w,f(w)) : w \in W\}$. $(E,V)$ is a tiling set because $|E| |V| = |W| |V| = p^d$ and the translates $E+v$ are disjoint because
the elements of $E+v$ are of the form $(w,f(w)+v)$ for some $w \in W$ and so $(w,f(w)+v) = (w',f(w')+v')$ implies (by looking at the first and then at the second coordinate) that $w=w'$ and $v=v'$.
\end{proof}

\subsection{Subspaces as spectral and tiling partners}

The proposition above classifies tiling sets with subspace tiling partners. We now prove that such sets are also spectral sets and a converse.
\begin{proposition}
\label{prop:subsptilspectr}
Let $V$ be a subspace of $\bb{Z}_p^d$, $V^\perp$ the orthogonal complement of $V$, and $E \subseteq \bb{Z}_p^d$. $(E,V)$ is a tiling pair if and only if $(E,V^\perp)$ is a spectral pair. This happens exactly when $E$ is the graph of a function $f: W \to V$ for some complement $W$ of $V$.
\end{proposition}

Prior to proving the proposition we give a proof of a well known fact about the Fourier transform of characteristic sets of subspaces.
\begin{lemma}
\label{Fouriersubspace}
Let $V$ be a subspace of $\bb{Z}_p^d$ and $V^\perp$ its orthogonal complement. Writing $V$ for the characteristic function of the subspace, we have
\[
\widehat{V}(m) = \begin{cases} \frac{|V|}{p^d} \;\; , & m \in V^\perp \\ 0 \;\; , &  m \notin V^\perp \end{cases}.
\] 
\end{lemma}
\begin{proof}
If $m \in V^\perp$, then $m \cdot v = 0$ for all $v \in V$ and so
\[
\widehat{V}(m) = \frac{1}{p^d} \sum_{v\in V} \chi(-m\cdot v) = \frac{|V|}{p^d}.
\]
If $m \notin V^\perp$, then $V$ is not contained in the hyperplane $\{x \in \bb{Z}_p^d : x \cdot m = 0 \}$. We will show that in this case $V$ equidistributes on the $p$ parallel hyperplanes $H_{m,t} = \{ x \, | \, x \cdot m = t \}$, $t=0,1,2,\dots,p-1$. Lemma \ref{lemma:main} implies that $\widehat{V}(m)=0.$

Let 
\[
V_t = \{ v \in V \, | \,  v \cdot m =t\}
\] 
be the intersection of $V$ with the hyperplane $H_{m,t}$. By our assumption $V\neq V_0$, so there exist $0\neq t \in \bb{Z}_p$ and $v_t \in V$ such that $v_t \cdot m = t$. It follows immediately that $V_t = v_t + V_0$ and so that $|V_t|=|V_0|$. Similarly, $V_{2t} = 2v_t + V_0$, which in turn implies $|V_{2t}|=|V_0|$. Repeating the same argument shows that $|V_{it}|=|V_0|$ for all $i \in \bb{Z}_p$, which is equivalent to $|V_t|$ being constant for all $t \in \bb{Z}_p$.  
\end{proof}

We now give a proof of the above proposition.

\begin{proof}[Proof of Proposition \ref{prop:subsptilspectr}]
Suppose first that $(E,V)$ is a tilling pair. We know from Theorem \ref{theorem:basicpicturetiling} that $|E| |V| = p^d$ and that $\widehat{E}(x) \widehat{V}(x)=0$ for all $0 \neq x \in \bb{Z}_p^d$. From this we deduce that $|V^\perp|=|E|$ and that $\widehat{E}(x-x')=0$ for all distinct $x,x\in V^\perp$. By Theorem \ref{theorem:basicpicturespectral}, we then have that $(E,V^\perp)$ is a spectral pair.

The first property is straightforward. $|V^\perp| = p^d / |V| = |E|.$ For the second observe that $V^\perp$ is a subspace and so $x-x' \in V^\perp$. The lemma above implies that $\widehat{V}(x-x') \neq 0$ and so we must have $\widehat{E}(x-x') = 0$.

Conversely now, suppose that $(E,V)$ is a spectral pair. We know from Theorem \ref{theorem:basicpicturespectral} that $|E| = |V|$ and that $\widehat{E}(x-x') =0$ for all distinct $x, x' \in V$. From this we deduce that $|V^\perp| |E| = p^d $ and that $\widehat{E}(x) \widehat{V^\perp}(x)=0$ for all  $0 \neq x\in \bb{Z}_p^d$. By Theorem \ref{theorem:basicpicturetiling}, we get that $(E,V^\perp)$ is a tiling pair.

The first property is once again straightforward. $|V^\perp| |E|  = (p^d / |V|) |E| = p^d.$ For the second observe that if $x \notin V=(V^\perp)^\perp$, then the lemma above implies $V^\perp(x)=0$; if on the other hand $0 \neq x \in V$, then $\widehat{E}(x) = \widehat{E}(x-0)=0$.
\end{proof}

\subsection{Spectral and tiling sets of dimension or codimension one}

The last task of this section is to prove that if a set has size $p$ or $p^{d-1}$, then the properties of being tilling set and a spectral set are equivalent.
\begin{proposition}
\label{prop:porp^d}
Let $E \subseteq \bb{Z}_p^d$ be a set of size $p$ or $p^{d-1}$. $E$ is a spectral set if and only if it is a tiling set. Furthermore this happens if and only if 
the tiling or spectral partner can be chosen to be a subspace and so happens if and only if $E$ is a graph.
\end{proposition}

\begin{proof}
There are four implications to consider.

\underline{$|E|=p$}. (i) Suppose first that $E$ is a tiling set. We show that $E$ tiles with a hyperplane through the origin  i.e., that there exists $0 \neq v \in \bb{Z}_p^d$ such that $(E,\text{span}(v)^\perp)$ is a tiling pair. Proposition \ref{prop:subsptilspectr} then implies that $E$ is a spectral set, whose spectrum is the span of $v$.

Let $A$ be a tiling partner for $E$. By Theorem \ref{theorem:basicpicturetiling} $\widehat{E}(v) \widehat{A}(v)=0$ for all $0 \neq v \in \bb{Z}_p^d$. Observe that $\widehat{A}$ is not identically zero on $\bb{Z}_p^d - \{0\}$ because $A \neq \bb{Z}_p^d$. So there exists $v \neq 0$ such that $\widehat{A}(v) \neq 0$ and so $\widehat{E}(v)=0$. Let $V = \text{span}(v)^\perp$. If $0\neq m \in V^\perp = \text{span}(v)$, then Lemma \ref{lemma:main} implies that $\widehat{E}(m)=0$; while if $m \notin V^\perp$, then Lemma \ref{Fouriersubspace} implies that $\widehat{V}(m) = 0$. So the product $\widehat{E} \, \widehat{V}$ is zero on $\bb{Z}_p^d - \{0\}$. Moreover, $|E| |V| = p \, p^{d-1} = p^d$. This proves that $(E,V)$ is a tiling pair as desired and we are done in this case.

(ii) Conversely now, suppose that $(E,B)$ is a spectral pair. We show the existence of $b \neq 0$ such that $(E,\text{span}(b)^\perp)$ is a tiling pair. 

By Corollary \ref{corr:spectral} we may translate $B$ so it contains 0. $|B|=|E|=p>1$ and so there exists $0 \neq b \in B$. For any such $b$ we have $\widehat{E}(b)=\widehat{E}(b-0) = 0$. Let $V = \text{span}(b)^\perp$. $(E,V)$ is a tiling pair because it satisfies the two standard properties. Firstly, $|E| |V| = p \, p^{d-1} = p^d$. Secondly, for every $0 \neq m$, either $m \in \text{span}(b)$ (in which case Lemma \ref{lemma:main} implies that $\widehat{E}(m)=0$) or $m \notin \text{span}(b)$ (in which case Lemma \ref{Fouriersubspace} implies $\widehat{V}(m)=0$). So $\widehat{E} \, \widehat{V} = 0$ on $\bb{Z}_p^d-\{0\}$. 

\underline{$|E|=p^{d-1}$}. (i)  Suppose first that $E$ is a tiling set. We will show that $E$ tiles with a line $L$ through the origin as partner. Proposition \ref{prop:subsptilspectr} then implies that $E$ is a spectral set, whose spectrum is $L^\perp$.

Theorem \ref{theorem:basicpicturetiling} (g) implies that $\text{DirC}(E)$, the direction cone of $E$, is not the whole of $\bb{Z}_p^d-\{0\}$. So there exists a line $L$ that is disjoint from it. It then follows from Theorem~\ref{theorem:basicpicturetiling}, that $(E,L)$ is a tiling pair as $\text{DirC}(E) \cap \text{DirC}(L) = \emptyset$.

(ii) Conversely now, suppose that $(E,B)$ is a spectral pair. We show there exists a line $L$ such that $(E,L)$ is a tiling pair.

Corollary \ref{corr:spectral} implies that $(B,E)$ is a spectral pair. Theorem \ref{theorem:basicpicturespectral} (d) implies that $\widehat{B}=0$ on $\text{\text{DirC}}(E)$. As $\widehat{B} \neq 0$ on the whole of $\bb{Z}_p^d - \{0\}$ (else $B=\bb{Z}_p^d$ contradicting $|B|=|E|=p^{d-1})$, there exists a line $L$ disjoint from $\text{DirC}(E)$. $(E,L)$ is a tiling pair following the same argument used in (i).

Note that in all four cases we showed that $E$ has a subspace as a tiling partner. Proposition \ref{prop:graph=tile} implies that $E$ is a graph. So all four hypotheses imply that $E$ is a graph. 
\end{proof}

\subsection{Spectral sets either tile or $k$-tile}

\begin{definition}
Fix $k \geq 1$ A subset $E$ of $\mathbb{Z}_p^d$ is said to $k$-tile with partner $A$ if every vector $\vec{x} \in \mathbb{Z}_p^d$ can be expressed as a sum of an element of $E$ with an element of $A$ in exactly $k$ ways. Thus for example $E$ $1$-tiles if and only if $E$ tiles in the manner previously discussed in this 
paper.

The pair $(E,A)$ is said to be a $k$-tiling pair and this happens if and only if $E \star A = k \mathfrak{1}$ where here $\star$ stands for the discrete convolution and 
$E, A$ now denote the characteristic functions of the corresponding sets. 

As before it is easy to see that this is equivalent to the conditions $\widehat{E}(m)\widehat{A}(m)=0$ for $m \neq 0$ and $|E||A|=kp^d$.
\end{definition}

The reader is warned that the concept of $k$-tiling is pretty weak if $k$ is not constrained. In fact every subset $E \subseteq \mathbb{Z}_p^d$ 
$|E|$-tiles with the whole space as partner. Thus it is only interesting that $(E,A)$ $k$-tiles when $1 \leq k < |E|$ or when the $k$-tiling partner 
is a proper subset of $\mathbb{Z}_p^d$.

We now prove:

\begin{theorem}
\label{theorem:ktiling}
A spectral set $E \subseteq \mathbb{Z}_p^d$ either $1$-tiles or it $\frac{|E|}{p}$-tiles with a hyperplane partner.
\end{theorem}
\begin{proof}
Let $(E,B)$ be a spectral pair in $\mathbb{Z}_p^d$. If $|E|=1$ then $E$ is a singleton set and hence 1-tiles so assume $|E|=|B| > 1$.
Thus $\widehat{E}(u)=0$ for some nonzero $u$ of the form $b-b'$ with $b, b' \in B$ distinct by Theorem~\ref{theorem:basicpicturespectral}. 
Letting $H$ denote the hyperplane through the origin perpendicular to $u$, we see that $\widehat{H}$ is supported on the line through $u$ on which 
$\widehat{E}$ vanishes away from the origin. Thus $\widehat{H}(m)\widehat{E}(m)=0$ for all nonzero $m$. As $|H||E|=\frac{|E|}{p} p^d$ we conclude 
that $(E,H)$ is a $\frac{|E|}{p}$-tiling pair as claimed.
\end{proof}

This has the following interesting corollary:

\begin{corollary}
\label{corollary:projection}
Let $E \subseteq \mathbb{Z}_p^d$ be a spectral set. Then either $E$ is a tiling set or there exists $m=\frac{|E|}{p} \leq p^{d-1}$ and a 
tiling set $E' \subseteq \mathbb{Z}_p^d \times \mathbb{Z}_m$ such that projection $\pi: \mathbb{Z}_p^d \times \mathbb{Z}_m \to \mathbb{Z}_p^d$ 
takes $E'$ bijectively to $E$. Thus every spectral set is either a tiling set or is the bijective projection of a tiling set.
\end{corollary}
\begin{proof}
By Theorem~\ref{theorem:ktiling} either $E$ tiles (in which case we are done) or it $m=\frac{|E|}{p}$-tiles with a hyperplane partner $H$.

In this second case recall that $H$ can be taken to be the hyperplane through the origin perpendicular to nonzero vector $u$ with 
$\widehat{E}(u)=0$. If $H=H_0, H_1, \dots, H_{p-1}$ is a labelling of the $p$ parallel hyperplanes to $H$ then $E$ has exactly $m=\frac{|E|}{p}$ points in each 
$H_j$ as it equidistributes on this family. We can thus choose bijections $\theta_j: E \cap H_j \to \mathbb{Z}_m$ for $0 \leq j \leq p-1$. We 
can define $f: E \to \mathbb{Z}_m$ to be the unique function which restricts to $\theta_j: E \cap H_j \to \mathbb{Z}_m$ for $0 \leq j \leq p-1$.

Finally define $E' = \{ (e, f(e)) | e \in E \} \subseteq \mathbb{Z}_p^d \times \mathbb{Z}_m$. Clearly the projection $\pi$ takes $E'$ bijectively 
to $E$. View $H$ as a subset of $\mathbb{Z}_p^d \times \mathbb{Z}_m$ via inclusion with zero last coordinate. We claim that 
$(E',H)$ is a tiling pair in $\mathbb{Z}_p^d \times \mathbb{Z}_m$. Note $|E'|=|E|$ so $|E'||H|=|E|p^{d-1}=mp^d = |\mathbb{Z}_p^d \times \mathbb{Z}_m|$ 
so the order condition for tiling is satisfied.

Given $(\vec{x},t) \in \mathbb{Z}_p^d \times \mathbb{Z}_m$ we first note that $\vec{x}$ lies in a unique $H_j$. As $\theta_j: E \cap H_j \to \mathbb{Z}_m$ 
is a bijection there is a unique element $\vec{e} \in E$ such that $\vec{e} \in H_j$ and $f(\vec{e})=t$. Since $\vec{x}, \vec{e}$ lie in the same coset $H_j$ of $H$, we have 
$\vec{x}-\vec{e} = \vec{h} \in H$. Thus $(\vec{x},t) = (\vec{e},t) + (\vec{h},0) = (\vec{e},f(\vec{e})) + (\vec{h},0)$ is a sum of an element of $E'$ and $H$. This combined with the order condition shows 
that $(E', H)$ is a tiling pair in $\mathbb{Z}_p^d \times \mathbb{Z}_m$ as desired and the proof is complete.

\end{proof}

\section{Proof of Theorem~\ref{theorem:main1}}

\begin{itemize}
\item[(a)] is proved in the opening paragraph of Section \ref{sect:tiling}.
\item[(b)] is proved in Corollary \ref{corr:spectralsize}.
\item[(c)] is proved in Corollary \ref{corollary:multipleof4}.
\item[(d)] is proved in Proposition \ref{prop:subsptilspectr} and Proposition \ref{prop:graph=tile}.
\item[(e)] is proved in Proposition \ref{prop:porp^d}.
\item[(f)] In $\bb{Z}_p$ parts (a) and (b), and Corollary  \ref{corr:spectralsize} imply that the only spectral or tiling sets are singletons or $\bb{Z}_p$. We now turn our attention to $\bb{Z}_p^2$. Part (a) implies that a tiling set has size $1$ or $p$ or $p^2$. Corollary \ref{corr:spectralsize} and Corollary \ref{corollary:spectralgaphighend} imply that the same is true for spectral sets. Singletons and the whole of $\bb{Z}_p^2$ are both tiling and spectral sets.  Part (e) implies that a set of size $p$ is a tiling set if and only if it is a spectral set. In all cases the set is also a graph set.
\item[(g)] Part (a) implies that a tiling set has size $1$ or $p$ or $p^2$ or $p^3$. Singletons and the whole of $\bb{Z}_p^3$ are spectral sets. By Part (e) a tiling set of size $p$ or $p^2$ is a spectral set. Furthermore in all cases the set is also a graph set.
\item[(h)] is proved in Theorem~\ref{theorem:ktiling}.
\item[(i)] is proved in Corollary~\ref{corollary:projection}.
\end{itemize}

\section{Construction and rank determination of a $2p \times 2p$ log-Hadamard matrix over $\mathbb{Z}_p$ and proof of 
Theorem~\ref{theorem:main2}}

When $p$ is an odd prime, an example of a $2p \times 2p$ Hadamard matrix whose entries were $p$th roots of unity was constructed in \cite{B62}. 
In this section we analyze (a scaling) of the corresponding log-Hadamard matrix with entries in $\mathbb{Z}_p$ and show it has rank $d$ with $4 \leq d \leq 5$. 
This in turn implies that there exist spectral sets of size $2p$ (which hence cannot tile) in $\mathbb{Z}_p^d$ when $d \geq 5$. 
As this analysis also involves the check that this matrix is log-Hadamard we do that also for completeness as our notation differs significantly from that 
found in \cite{B62}.

First we provide the definition of the $2p \times 2p$ matrix $\mathbb{L}$ with entries in $\mathbb{Z}_p$. Let $p$ be and odd prime, $q=\frac{p-1}{2}$ and 
$n$ any nonsquare modulo $p$. Note $q, n \neq 0$ in $\mathbb{Z}_p$ and $2q=-1 \text{ mod } p$. Furthermore note that in the partition $\mathbb{Z}_p = \{ 0 \} \cup S \cup N$ into zero, the nonzero squares, and the non squares, multiplication by $n$ induces a bijection between the sets $S$ and $N$. In the following we will 
exclusively use $i$ for the index of rows of a matrix and $j$'s for the index of columns of a matrix. We will also start our indexing of rows and columns at 
zero so the zeroth row is the topmost row and the zeroth column the leftmost column of the matrix. Thus the indexes $i,j$ of our $2p \times 2p$ matrix range from 
$0$ to $2p-1$ though their mod $p$ values range through $0$ to $p-1$ twice consecutively.
We define
$$
\mathbb{L} = \begin{bmatrix} \mathbb{A} & n\mathbb{A} \\ \mathbb{B} & \mathbb{C} \end{bmatrix}
$$
where $\mathbb{A}_{ij} = 2ij-i^2, \mathbb{B}_{ij}=(j-ni)^2, \mathbb{C}_{ij}=n(i-j)^2$ are $p \times p$ matrices with entries in $\mathbb{Z}_p$ and 
$i,j$ range from $0$ to $p-1$. We will refer to the first $p$ rows of $\mathbb{L}$ as the ``top rows'' of $\mathbb{L}$ and the last $p$ rows as the ``bottom rows''.
Notice the topmost (0th) row of $\mathbb{L}$ is the zero (row) vector which will imply the other rows are balanced vectors once we verify that 
$\mathbb{L}$ is log-Hadamard.

\begin{lemma} 
\label{lemma:tophalf}
The matrix $\mathbb{A}$ is log-Hadamard. The difference of two distinct ``top rows'' of $\mathbb{L}$ is a balanced vector 
in the span of the two vectors \\
$\begin{bmatrix} 0, 1, 2, \dots, p-1 | 0 \cdot n, 1 \cdot n, 2 \cdot n, \dots, (p-1) \cdot n \end{bmatrix}$ and $\begin{bmatrix} 1, \dots, 1 | n, \dots, n \end{bmatrix}$
where we use a $|$ to denote the midpoint of the $2p$-dimensional vectors. In particular all the ``top rows'' of the matrix $\mathbb{L}$ lie in the span of these 
two vectors. 
\end{lemma}

\begin{proof}
We will use the notation $\begin{bmatrix} f(j) | g(j) \end{bmatrix}$ as shorthand for a $2p$-dimensional vector whose first $p$ coordinates are given by the formula 
$f(j)$ as the column index $j$ varies over $0, \dots, p-1$ and whose last $p$ coordinates are given by the formula $g(j)$ as the column index $j$ varies 
over $0, \dots, p-1$. Thus $[j | j]$ encodes the vector $[0, 1, 2, \dots, p-1 | 0, 1, 2, \dots, p-1]$ and $[1 | n]$ the vector $[1,1,\dots,1 | n, n, \dots, n]$ for example.

Then the difference of  ``top row'' $i$ and ``top row'' $i'$ of the matrix $\mathbb{L}$ is $[f(j) | n f(j)]$ where 
$f(j)= 2ij-i^2-(2i'j-(i')^2) = 2(i-i')j-(i^2-(i')^2)$. This means that it is equal to $2(i-i')[j | nj] - (i^2-(i')^2)[1|n]$ and so lies in the span of the two vectors 
$[j | nj]$ and $[1|n]$ as claimed. Finally as $(i-i')$ is nonzero, $(i-i')[j|nj]$ has both halves balanced $p$-dimensional vectors and adding any multiple of 
$[1 | n]$ to this does not change this. So the difference of these two distinct top rows is balanced as claimed.  

Finally as the topmost row of $\mathbb{L}$ is the zero vector, any other ``top row'' of $\mathbb{L}$ can be viewed as the difference of itself and the topmost row 
and so lies in the span of the vectors $[1|n]$ and $[j | nj]$.
\end{proof}

\begin{lemma}
\label{lemma:bottomhalf}
The difference of two distinct ``bottom rows" of $\mathbb{L}$ is a balanced vector in the span of the two vectors 
$[j|j]$ and $[n|1]$. Thus the ``bottom rows'' lie in the span of the three vectors $[j|j], [n|1]$ and $[j^2|nj^2]$. 
\end{lemma}
\begin{proof}
The difference between the $i$th ``bottom row" and the $i'$th ``bottom row" is \\ $[(j-ni)^2-(j-ni')^2 | n(i-j)^2-n(i'-j)^2]$. This simplifies to 
$[n(i'-i)j-n^2((i')^2-i^2) | n(i'-i)j - n((i')^2-i^2) ] =n(i'-i)[j|j]-n((i')^2-i^2)[n | 1]$. Thus the difference of two distinct ``bottom rows" lies in the span 
of the two vectors $[j|j]$ and $[n|1]$. Furthermore as $n(i'-i)$ is nonzero, $n(i'-i)[j|j]$ is balanced (in each half) and adding any multiple of $[n|1]$ to this 
does not change this which establishes the first statement of the lemma. The second statement follows as any ``bottom row" can be written as 
the 0th ``bottom row'' $[j^2 | nj^2] = [j^2 | nj^2]$ plus the difference of itself and the 0th ``bottom row'' which is a linear combination of 
$[n|1]$ and $[j|j]$.
\end{proof}

\begin{corollary}
\label{corollary:Hadamardrank}
If $d$ is the rank of $\mathbb{L}$ then $4 \leq d \leq 5$. The rowspace of $\mathbb{L}$ lies in the span of the $5$ vectors 
$[j|nj],[1|n],[j|j],[n|1]$ and $[j^2|nj^2]$.
\end{corollary}
\begin{proof}
The last part follows immediately from Lemmas~\ref{lemma:tophalf} and ~\ref{lemma:bottomhalf}. From this it follows that $d \leq 5$.
As $n \neq 1$, $[j|j]$ and $[j,nj]$ are linearly independent and as the span of these two vectors consists of vectors whose leftmost coordinate is zero, 
we see that $[1|n]$ is independent from these two. Finally, to see that $[j^2 | nj^2]$ is independent from the set of three vectors $\{ [j|j],[j|nj],[1|n] \}$, 
note that in the first half of the any linear combination of these three vectors is a linear function of $j$ which cannot equal the quadratic function $j^2$ 
as $f(j)=j^2$ has image size $\frac{p+1}{2}$ as a function $\mathbb{Z}_p \to \mathbb{Z}_p$ whereas a linear function has image size $p$ or $1$. Thus $d \geq 4$ always.
\end{proof}

\begin{corollary}
\label{corollary:Hadamardrank2}
When $p \equiv 3$ mod $4$ and we select $n=-1$, then the rank of $\mathbb{L}$  is four.
\end{corollary}
\begin{proof}
Note $-1$ is a valid choice for $n$ if and only if $-1$ is not a square in $\mathbb{Z}_p$ if and only if $p \equiv 3$ mod $4$.
In this case $[1|n]$ and $[n|1]$ are scalar multiples of each other so the result follows by Corollary~\ref{corollary:Hadamardrank}.
\end{proof}

To finish the check that $\mathbb{L}$ is log-Hadamard in light of lemmas~\ref{lemma:tophalf} and \ref{lemma:bottomhalf} it remains to show that the 
difference of a ``top row" of $\mathbb{L}$ with a ``bottom row'' of $\mathbb{L}$ is a balanced vector. This will require some preliminary considerations 
which we do next.

We first record a useful lemma about the image set of a quadratic map $Q: \mathbb{Z}_p \to \mathbb{Z}_p$:

\begin{lemma}
\label{lemma:quadimage}
Fix $p$ and odd prime. Let $Q(x)=ax^2+bx+c: \mathbb{Z}_p \to \mathbb{Z}_p$ be a quadratic map with $a \neq 0$.
Then for all $\mu \in \mathbb{Z}_p$ we have: \\
$$
|Q^{-1}(\mu)|=\begin{cases} 2 \text{ if } b^2-4a(c-\mu) \text{ is a nonzero square } \\
1 \text{ if } b^2-4a(c-\mu)=0 \\
0 \text{ if } b^2 -4a(c-\mu) \text{ is a nonsquare }
\end{cases}
$$
\end{lemma}
\begin{proof}
Follows immediately from the quadratic formula.
\end{proof}

\begin{definition}
Fix $p$ an odd prime. A pair of quadratic functions $Q_1, Q_2 \in \mathbb{Z}_p[x]$ is a balanced pair if 
for every $\mu \in \mathbb{Z}_p$, $|Q_1^{-1}(\mu)| + |Q_2^{-1}(\mu)|=2$. In this case the $2p$-dimensional vector 
$( Q_1(j) | Q_2(j) )$ is a balanced vector.
\end{definition}

By Lemma~\ref{lemma:quadimage} it follows that if $Q_1(x)=ax^2+bx+c, Q_2(x)=\tilde{a}x^2 + \tilde{b}x + \tilde{c}$ are two quadratic maps 
in $\mathbb{Z}_p[x]$ with $a\tilde{a} \neq 0$ and $b^2-4ac= n[\tilde{b}^2 - 4\tilde{a}\tilde{c}], a=n\tilde{a}$ for some nonsquare $n$ then the pair 
$(Q_1, Q_2)$ is a balanced pair of quadratic functions.

We are now ready to complete the verification that $\mathbb{L}$ is log-Hadamard.

If we take the difference between the $i$th ``top row vector and $i'$ ``bottom row'' vector of $\mathbb{L}$ (note $i=i'$ is possible) we get the row vector
$$
[2ij-i^2-(j-ni')^2 | n(2ij-i^2)-n(i'-j)^2] 
$$
which simplifies to
$$
[-j^2 +(2ni'+2i)j + (-i^2-n^2(i')^2) | -nj^2 + n(2i'+2i)j + (-ni^2-n(i')^2)]=[Q_1(j)|Q_2(j)]
$$
where $Q_1(j)=aj^2+bj+c, Q_2(j)=\tilde{a}j^2 + \tilde{b}j + \tilde{c}$ with $a=-1, \tilde{a}=-n, b=(2ni'+2i), \tilde{b}=n(2i'+2i), c=(-i^2-n^2(i')^2), \tilde{c}=-n(i^2+(i')^2)$.
It is then easy to verify that $b^2-4ac=n(\tilde{b}^2-4\tilde{a}\tilde{c}), a=n\tilde{a}$ and so it follows that $(Q_1, Q_2)$ is a balanced pair of quadratics 
and so the difference vector is a balanced vector.

We have  thus finished verifying that $\mathbb{L}$ is log-Hadamard and thus we have proven Theorem~\ref{theorem:main2}: 

\begin{theorem}
For any odd prime $p$, there exists a spectral set of order $2p$ in $\mathbb{Z}_p^5$ which does not tile. 
For primes $p$ with $p \equiv 3 \text{ mod } 4$, there exist spectral sets of order $2p$ in $\mathbb{Z}_p^4$ which do not tile.
\end{theorem}
\begin{proof}
Follows by Corollaries~\ref{corollary:Hadamardrank} and ~\ref{corollary:Hadamardrank2} and Theorem~\ref{theorem:logHadamardSpectralEquivalence}.
\end{proof}

As an illustration we provide the sets $E$ and $B$ for the case when $p \equiv 3 \text{ mod } 4$ and $n=-1$. Both are subsets of $\bb{Z}_p^4$. Label the elements of $E$ by $E=\{e_0,\dots,e_{2p-1}\}$ and the elements of $B$ by $B=\{b_0,\dots,b_{2p-1}\}$. Now letting $i$ and $j$ range from $0$ to $p-1$ we have
\[
e_i = (i^2,0,2i,0) , \, e_{i+p} = (-i^2,2i,0,1) \text{ and } b_{j} = (-1,j,-j,j^2), \, b_{j+p} = (1,j,-j,-j^2),
\]
with all the entries being elements of $\bb{Z}_p$.

\section{3-dimensional Fuglede conjecture over prime fields}

One of the important remaining questions is the status of the Fuglede conjecture in 3-dimensions over prime fields. It has the following equivalent formulations:

\begin{proposition}[Equivalent formulations of 3-dimensional Fuglede conjecture]
\label{theorem:3DFuglede}
Let $p$ be a prime. The following are equivalent: 

\vskip.125in 

(a) A subset of $\mathbb{Z}_p^3$ tiles if and only if it is spectral. 

(b) There do not exist spectral sets $E \subseteq \mathbb{Z}_p^3$ of size $mp$, $1 < m < p$. 

(c) There does not exist a $mp \times mp$ log-Hadamard matrix with entries in $\mathbb{Z}_p$ of 

rank $\leq 3$ where $1 < m < p$. 

(d) There does not exist a $mp \times mp$ log-Hadamard matrix with entries in $\mathbb{Z}_p$ of rank $3$ 

where $1 < m < p$. 

(e) There does not exist a $mp \times mp$ special dephased log-Hadamard matrix with entries 

in $\mathbb{Z}_p$ of rank $3$ where $1 < m < p$. 
\end{proposition}

\vskip.125in 

\begin{proof}
By Theorem~\ref{theorem:main1}, a tiling set in $\mathbb{Z}_p^3$ is always spectral. Furthermore spectral sets of size $1, p, p^2,p^3$ in $\mathbb{Z}_p^3$, 
always tile. Thus by the same theorem the only spectral sets that do not tile would have sizes $mp, 1 < m < p$ and any such set definitely does not tile by order considerations. Thus $(a)$ is equivalent to $(b)$. $(b)$ is equivalent to $(c)$ by Theorem~\ref{theorem:logHadamardSpectralEquivalence}. 
$(c)$ is equivalent to $(d)$ as there do not exist spectral sets of size $mp, 1 < m < p$ in $\mathbb{Z}_p^2$ by Theorem~\ref{theorem:main1}. The equivalence 
of $(d)$ and $(e)$ is given by Corollary~\ref{corollary:dephasedrank}.
\end{proof}

We now provide a human-readable proof of the 3-dimensional Fuglede conjecture when $p=2,3$.

\begin{theorem}
\label{theorem:3Dverifications}
The Fuglede conjecture holds in $\mathbb{Z}_2^3$ and $\mathbb{Z}_3^3$.
\end{theorem}
\begin{proof}
When $p=2$, condition (b) of Proposition~\ref{theorem:3DFuglede} holds and so we are done in this case.

Now consider $p=3$. By Proposition~\ref{theorem:3DFuglede} it is enough to rule out the existence of a special dephased $6 \times 6$ log-Hadamard matrix with $\mathbb{Z}_3$ entries of rank 3.
We will use our previous computation of the Davey monoid of $3 \times 3$ Davey matrices and the corresponding triplet rule to show that 
up to permutation of rows and columns there is a {\bf unique} special dephased $6 \times 6$ log-Hadamard matrix with $\mathbb{Z}_3$ entries.
This matrix will have rank 4 and so it will follow that no rank 3, special dephased $6 \times 6$ log-Hadamard matrix with $\mathbb{Z}_3$ entries 
exists hence proving the Fuglede conjecture in $\mathbb{Z}_3^3$.

Let $\mathbb{A}$ be a special dephased log-Hadamard $6 \times 6$ matrix with $\mathbb{Z}_3$ entries. Thus 
$$
\mathbb{A}=\begin{bmatrix} 0 & 0 & 0 & 0 & 0 & 0 \\ 0 & 1 & 2 & 0 & 1 & 2 \\ 0 & 2 & * & * & * & * \\ 0 & 0 & * & * & * & * \\
0 & 1 & * & * & * & * \\ 0 & 2 & * & * & * & * \end{bmatrix}
$$
We will refer to the rows as 0th through 5th row (0th row on top) and columns 0th through 5th (0th on the left) throughout.

Now the 1st row and the 3rd row are two balanced 6-dimensional vectors whose difference is balanced so the triplet rule says that as a 0 occurs in the 3rd row below a 1 in the 1st row, we must also have a 1 occur in the 3rd row below a 0 in the 1st row and a 2 occur in the 3rd row below a 2 in 1st row. 
Thus $A_{3,3}=1$ and one of $A_{3,2}, A_{3,5}$ equals two. By permuting the 2nd and 5th columns if necessary we can assume $A_{3,2}=2$. 
(We will not be permuting columns anymore after this.)

Thus the 3rd row looks like $(0,0,2,1,*,*)$. As there is a $0$ in the 3rd row below a $0$ in the 1st row, the triplet rule forces that there must be a $2$ 
in the 3rd row below a $1$ in the 1st row and a $1$ in the 3rd row below a $2$ in the 1st row. This forces the 3rd row to be 
$(0,0,2,1,2,1)$. Thus 
$$
\mathbb{A}=\begin{bmatrix} 0 & 0 & 0 & 0 & 0 & 0 \\ 0 & 1 & 2 & 0 & 1 & 2 \\ 0 & 2 & * & * & * & * \\ 0 & 0 & 2 & 1 & 2 & 1 \\
0 & 1 & * & * & * & * \\ 0 & 2 & * & * & * & * \end{bmatrix}
$$

Now by considering the $1$st and $2$nd column which are balanced, with balanced difference, we see that there is exactly one occurrence of a 
zero to the left of a zero. Thus the triplet rule guarantees the occurrence of exactly one 1 to the right of a 2. By permuting the 2nd and 5th rows if necessary 
(We will not be permuting rows anymore after this.) we can assume $A_{2,2}=1$. Then the occurrence of a 2 in the 2nd column to the right of a 0 in the 1st column forces the occurrence of a 0 to the right of a 2 and a 1 to the right of a 1. This forces the form of the 2nd column to be as below: 

$$
\mathbb{A}=\begin{bmatrix} 0 & 0 & 0 & 0 & 0 & 0 \\ 0 & 1 & 2 & 0 & 1 & 2 \\ 0 & 2 & 1 & * & * & * \\ 0 & 0 & 2 & 1 & 2 & 1 \\
0 & 1 & 1 & * & * & * \\ 0 & 2 & 0 & * & * & * \end{bmatrix}
$$

Now when comparing the 4th and 5th rows with the 1st row, three pairs of entries are known and the triplet rule forces the remaining unknown entries (no permutations or ambiguity) as the reader can verify resulting in the matrix:

$$
\mathbb{A}=\begin{bmatrix} 0 & 0 & 0 & 0 & 0 & 0 \\ 0 & 1 & 2 & 0 & 1 & 2 \\ 0 & 2 & 1 & * & * & * \\ 0 & 0 & 2 & 1 & 2 & 1 \\
0 & 1 & 1 & 2 & 2 & 0 \\ 0 & 2 & 0 & 2 & 1 & 1 \end{bmatrix}
$$

The remaining entries of $\mathbb{A}$ are then forced as the last 3 columns must be balanced and we get: 
$$
\mathbb{A}=\begin{bmatrix} 0 & 0 & 0 & 0 & 0 & 0 \\ 0 & 1 & 2 & 0 & 1 & 2 \\ 0 & 2 & 1 & 1 & 0 & 2 \\ 0 & 0 & 2 & 1 & 2 & 1 \\
0 & 1 & 1 & 2 & 2 & 0 \\ 0 & 2 & 0 & 2 & 1 & 1 \end{bmatrix}
$$

A quick Gauss-Jordan elimination reveals this matrix to have rank 4 and thus any special dephased $6 \times 6$ log-hadamard matrix with entries in 
$\mathbb{Z}_3$ has rank 4 (not 3) and we are done.

\end{proof}

Though counterexamples to the Fuglede conjecture exist in 3-dimensions over some non prime cyclic rings, these are of the form 
tiling sets which are not spectral. As we know tiling sets are always spectral in $\mathbb{Z}_p^3$ when $p$ is prime and so these examples cannot help decide 
whether the Fuglede conjecture is true or not in 3-dimensons over prime fields. We have seen that generally Fuglede is true in 2 dimensions and false in 4 dimensions over prime cyclic rings so 3-dimensions remains the last remaining significant case of this conjecture over prime rings.

Unfortunately the Davey monoid methods and computer methods seem to become computationally infeasible as the prime $p$ grows. 
In $\mathbb{Z}_5^3$ for example one has to rule out the existence of spectral sets of size $10, 15, 20$ in order to prove the conjecture. 
The number of subsets of these sizes are $\binom{125}{10}+\binom{125}{15} + \binom{125}{20}$ which is large and furthermore many spectral partners 
must be considered even after some simplications. The set of $5 \times 5$ Davey matrices seems also hard to determine as $5!=120$ permutations must be considered in order to understand it.

Another question left open is the question of whether tiling sets are always spectral over prime cyclic rings in any dimension. Surprisingly this appears to remain a possibility though many counterexamples exist over non prime cyclic rings.

\enddocument